\documentclass[reqno]{amsart}
 \usepackage{amssymb, amsmath}
 \usepackage{mathrsfs}
 \usepackage{mathtools}
 \usepackage{tikz}
 \usetikzlibrary{matrix,positioning, arrows}
 \usepackage{enumerate}
 \usepackage{hyperref,adjustbox}

\renewcommand{\leq}{\leqslant}
\renewcommand{\geq}{\geqslant}
\newcommand{\seq}{\subseteq}

\newcommand\restr[2]{{
  \left.\kern-\nulldelimiterspace 
  #1 
  \vphantom{\big|} 
  \right|_{#2} 
  }}

\makeatletter
\newdimen\rh@wd
\newdimen\rh@hta
\newdimen\rh@htb
\newbox\rh@box
\def\rh@measure#1{\setbox\rh@box=\hbox{$#1$}\rh@wd=\wd\rh@box \rh@hta=\ht\rh@box}

\def\widecheck#1{\rh@measure{#1}%
  \setbox\rh@box=\hbox{$\widehat{\vrule height \rh@hta width\z@ \kern\rh@wd}$}%
  \rh@htb=\ht\rh@box \advance\rh@htb\rh@hta \advance\rh@htb\p@
  \ooalign{$\vrule height \ht\rh@box width\z@ #1$\cr
           \raise\rh@htb\hbox{\scalebox{1}[-1]{\box\rh@box}}\cr}}
\makeatother

\newtheorem{theorem}{Theorem}[section]
\newtheorem{lemma}[theorem]{Lemma}
\newtheorem{corollary}[theorem]{Corollary}
\newtheorem{proposition}[theorem]{Proposition}

\theoremstyle{definition}

\newtheorem{definition}[theorem]{Definition}

\newtheorem{remark}[theorem]{Remark}
\newtheorem{remarks}[theorem]{Remarks}

\newtheorem{assumption}{Assumption}

\renewcommand{\paragraph}[1]{\textit{#1}}

\newcommand{\Cx}{\ensuremath{\mathbf{C}}} 
\newcommand{\Cst}{\ensuremath{\sf C^*}} 
\newcommand{\T}{{\sf T}} 
\newcommand{\I}{\mathscr{I}} 
\renewcommand{\S}{{\sf S}} 
\newcommand{\D}{{\sf D}} 
\DeclareMathOperator{\CC}{\mathbb{C}} 
\newcommand{\C}{\mathscr{C}} 
\DeclareMathOperator{\VV}{\mathbb{V}} 
\newcommand{\V}{\mathscr{V}} 
\DeclareMathOperator{\VC}{\VV\circ\CC} 
\DeclareMathOperator{\Eq}{\rm Eq} 
\DeclareMathOperator{\dom}{\mathrm{dom}} 
\DeclareMathOperator{\id}{\mathrm{id}} 
\newcommand{\sprod}[1]{\underset{#1}\smallproduct} 
\DeclareMathOperator{\smallproduct}{\Pi} 
\DeclareMathOperator{\Sub}{Sub}

\newcommand{\Q}{{\sf Q}} 
\newcommand{\A}{{\sf X}} 
\newcommand{\B}{{\sf Y}} 
\newcommand{\Cc}{{\sf C}} 
\newcommand{\Set}{{\sf Set}} 
\newcommand{\F}{{\mathscr{F}}} 
\newcommand{\U}{{\mathscr{U}}} 
\newcommand{\f}{\varphi} 
\newcommand{\Va}{{\sf V}} 
\newcommand{\Vap}{{\sf V}_{\sf p}} 
\renewcommand{\a}{{\triangle}} 
\renewcommand{\b}{{\bigtriangledown}} 
\newcommand{\ev}{{\rm ev}} 

\title{General affine adjunctions, Nullstellens{\"a}tze, and dualities}
\author[O. Caramello]{Olivia Caramello}
\address[O. Caramello]{Dipartimento di Scienza e Alta Tecnologia, Universit\`a degli Studi dell'Insubria, Via Valleggio 11, 22100 Como, Italy.}
\email{olivia@oliviacaramello.com}

\author[V. Marra]{Vincenzo Marra}
\address[V. Marra]{Dipartimento di Matematica \textsl{Federigo Enriques}, Universit\`a degli Studi di Milano, Via Cesare Saldini 50, 20133 Milano, Italy.}
\email{vincenzo.marra@unimi.it}

\author[L. Spada]{Luca Spada}
\address[L. Spada, corresponding author]{Dipartimento di Matematica, Universit\`a degli Studi di Salerno, Via Giovanni Paolo II 132, 84084 Fisciano (SA), Italy.}
\email{lspada@unisa.it}


\begin{document}
\begin{abstract}
We introduce and investigate a category-theoretic abstraction of the standard ``system-solution'' adjunction in affine algebraic geometry. We then look further into these geometric adjunctions at different levels of generality, from syntactic categories to (possibly infinitary) equational classes of algebras. In doing so, we discuss the relationships between the dualities induced by our framework and the well-established  theory of concrete dual adjunctions.  In the context of general algebra we prove an analogue of Hilbert's {\it Nullstellensatz}, thereby achieving a complete characterisation of the fixed points on the algebraic side of the adjunction. 
\end{abstract}

\maketitle

\section{Introduction}\label{s:intro}
We are concerned in this paper with an abstraction of the  affine adjunction of classical algebraic geometry to  algebraic and categorical settings. This introduction describes our results in increasing order of generality. For illustrative purposes we begin with the  case of $k$-algebras and proceed through successive generalisations;  the article, by contrast, develops the general case first.

 If $k$ is an algebraically closed field, any subset $R$ of the polynomial ring over finitely many variables $k[X]:=k[X_{1},\ldots,X_{n}]$ determines the (possibly infinite) system of equations:
\begin{equation}\label{eq:sist}
p(X_1,\ldots, X_{n})=0, \ p \in R.
\end{equation}
Let us write $\VV{(R)}\subseteq k^{n}$ for the set of solutions of \eqref{eq:sist} over $k^{n}$, where $k^{n}$ is the affine $n$-space over $k$. 
Then $\VV{(R)}$  
is the \emph{affine set defined   by $R$}. 

Conversely, for any subset $S\seq k^{n}$ we can consider the set $\CC{(S)}\seq k[X]$ of polynomials that vanish over $S$, which is automatically an ideal. Then $\CC{(S)}$ is the \emph{ideal defined by $S$}.   
Writing $2^{E}$ for the power set  of the set $E$, we obtain functions (implicitly indexed by $n$)
\begin{align}
\CC&\colon 2^{k^{n}}\longrightarrow 2^{k[X]},\label{intro:c}\\
\VV&\colon 2^{k[X]}\longrightarrow 2^{k^{n}}\label{intro:v}
\end{align}
that yield a (contravariant) Galois connection. The fixed points of the closure operator $\VV\circ\CC$ are then, by definition,  the affine sets in $k^{n}$. Since $\VV\circ\CC$ is a \emph{topological} closure operator ---i.e.\ it commutes with finite unions--- affine algebraic sets are  the closed sets of a topology on $k^{n}$, the  \emph{Zariski topology}. The fixed points of the dual closure operator $\CC\circ\VV$, on the other hand, are characterised by Hilbert's {\it Nullstellensatz}: they are precisely the \emph{radical ideals} of $k[X]$, that is, those ideals that coincide with the intersection of all prime ideals containing them. The {\it Nullstellensatz} thus characterises coordinate rings, for  $k[X]/I$ is one such if, and only if, $I$ is radical. Since radical ideals may in turn be elementarily characterised  as those ideals $I$ such that $k[X]/I$ has no non-zero nilpotents, coordinate rings are precisely the finitely presented nilpotent-free (or \emph{reduced}) $k$-algebras.

The Galois connection given by the pair $(\CC,\VV)$ in (\ref{intro:c}--\ref{intro:v}) can be made functorial. On the algebraic side we consider the category of finitely presented $k$-algebras with their homomorphisms. On the geometric side we take sets $S$ equipped with a specific embedding $S\hookrightarrow k^{n}$, as $n$ ranges over the natural numbers. It is important not to blur the distinction between  $S$ and its embedding into $k^{n}$, because arrows in the geometric category are to be defined \emph{affinely}, i.e.\ by restriction from $k^{n}$. An arrow from $S\hookrightarrow k^{n}$ to $T\hookrightarrow k^{m}$ is a \emph{regular map} $S\to T$, that is, the equivalence class of a polynomial function $f\colon k^{n}\to k^{m}$ such that $f$ throws $S$ onto $T$; two such functions are equivalent if, and only if, they agree on $S$. There is a functor that associates to each regular map $S\to T$ a contravariant homomorphism of the (automatically presented) coordinate rings of $\VV\circ\CC{(T)}$ and $\VV\circ\CC{(S)}$. And there is a companion functor that associates to each homomorphism of presented $k$-algebras $k[X]/I\to k[Y]/J$, with $Y=\{Y_1,\ldots,Y_{m}\}$ and $J$ an ideal of $k[Y]$, a contravariant regular map $\VV{(J)}\to\VV{(I)}$. The two functors yield a contravariant adjunction; upon restricting each functor to the fixed points in each domain, one obtains the classical duality (=contravariant equivalence) between affine algebraic varieties and their  coordinate rings. Compare\footnote{Terminology: Hartshorne's corollary is stated for irreducible varieties, which he calls varieties {\it tout court}.} e.g.\  \cite[Corollary 3.8]{hartshorne1977algebraic}. 

The classical affine adjunction above can be generalised to any \emph{variety of algebras}, whether finitary to infinitary. 
For background on Birkhoff's theory of general, or ``universal'',  algebra, see e.g.\ \cite{birkhoff79, cohn81,  jacobson80, Burris:81}. Henceforth, variety (of algebras) means ``possibly infinitary variety (of algebras)'' in the sense of S{\l}ominsky \cite{Slominsky:1959} and Linton \cite{Linton:1965} (after Lawvere \cite{lawverereprint}).
In a general variety the \emph{free algebras} play the same r\^{o}le as the ring of polynomials in the above correspondence. Congruences of  free algebras replace ideals of  rings of polynomials. The  ground field $k$ is replaced by an arbitrary but fixed algebra $A$ in the variety.  Beginning with Section \ref{sec:algebraic} we show that the classical affine adjunction for $k$-algebras  extends \textit{verbatim} to this general algebraic setting. 
\begin{remark}\label{r:presentability}By definition, the coordinate rings  $k[X]/I$  are \emph{presented}, that is,  they come with a specific defining ideal $I$.    For our purposes here the presented and the \emph{presentable} objects are to be kept distinct.  If $\Va$ is a variety, we indicate by $\Vap$ the category of presented algebras in the variety. Then $\Va$ and $\Vap$ are equivalent categories; cf.\ Remark \ref{rem:Vpequiv} below.
\end{remark}
 For any $\Va$-algebra $A$, in Corollary \ref{cor:algadj} we obtain an adjunction between $\Vap^{\rm op}$, the opposite of the category of presented $\Va$-algebras, and the category of subsets of $A^\mu$, as $\mu$ ranges over all cardinals, with  definable maps as morphisms. Here, the definable maps are those that are term-definable. The functors that implement the adjunction act on objects by taking a subset $R\seq \F{(\mu)}\times\F{(\mu)}$ (where $\F{(\mu)}$ is ``the'' free $\Va$-algebra on $\mu$ generators) ---that is, a ``system of equations in the language of $\Va$''---  to its solution set $\VV{(R)}\seq A^\mu$, where $\VV{(R)}$ is the set of elements of $A^\mu$ such that each pair of terms in $R$ evaluate identically over it, and a subset $S\hookrightarrow A^\mu$ to its ``coordinate  $\Va$-algebra'', namely, $\F{(\mu)}/\CC{(S)}$, where $\CC{(S)}$ is the congruence on $\F{(\mu)}$ consisting of all pairs of terms that evaluate identically at each element of $S$.

The identification of the fixed points of this adjunction on the algebraic side leads to a result formally analogous to the ring-theoretic {\it Nullstellensatz}. The final result is stated as Theorem \ref{thm:algnull}. Additionally, the identification of an appropriate type of representation for those $\Va$-algebras that are fixed under the adjunction (in part (iii) of the theorem) leads to a result reminiscent of Birkhoff's subdirect representation theorem. The notorious failure of the latter for infinitary varieties is irrelevant for our purposes here. In fact, while our Theorem \ref{thm:algnull} may  be conceived of as a  version of the subdirect representation theorem that is  ``relative to the ground algebra $A$'', it is formally incomparable to Birkhoff's result: neither statement entails the other, in general; we return to this important point in Remark \ref{r:finitary-vs-inifinitary}. We also discuss, in Section \ref{topNullstellensatz}, the problem of characterising the fixed points of the adjunction on the topological side, establishing some partial results which we subsequently use for deriving `presented versions' of Stone and Gelfand dualities from our framework. 

At the beginning of the paper, starting from Section \ref{s:affadj}, we show how to lift the adjunction of Corollary \ref{cor:algadj}  from the algebraic setting to a more general categorical context. See Table \ref{tab:transl} for a short summary of the correspondence between concepts in the three settings.    Conceptually, the key ingredient in the algebraic construction sketched above is the functor $\I_{A}=\hom_{\Va}(-, A) \colon \T \to \Set$ from the opposite $\T$ of the category of free $\Va$-algebras to $\Set$. In our categorical abstraction, the basic {\it datum} is any functor $\I \colon \T \to \S$, which we conceive as the \emph{interpretation} of the ``syntax'' $\T$ into the ``semantics'' $\S$,  along with a distinguished object $\a$ of $\T$. (In the algebraic specialisation, $\a$ is $\F{(1)}$, the free singly generated $\Va$-algebra.)  Here $\T$ and $\S$ are  arbitrary (locally small\footnote{All  categories in this paper are assumed to be locally small.}) categories, and  $\S$ is well-powered.\footnote{That is, each object of $\S$ has a small poset of subobjects.} Out of these ingredients we construct two categories $\D$ and $\Q$ of subobjects and quotients, respectively.  
The category $\D$ abstracts that of sets affinely embedded into $A^{\mu}$; here, sets are replaced by objects of $\S$, the powers $A^{\mu}$ are replaced by objects $\I(t)$ as $t$ ranges over objects of $\T$, and the  morphisms of $\S$ that are ``definable'' are declared to be  those in the range of $\I$.  The category $\Q$ abstracts the category of quotients of the free $\Va$-algebras $\F{(\mu)}$ by congruences; its objects are quotients of the hom-set $\hom_{\T}(t,\a)$, as $t$ ranges over objects of $\T$, by ``$\a$-congruences'' on them.  It is possible to define the operator $\VV$ in this setting under the hypothesis that $\S$ has enough limits (Assumption \ref{ass:limits} below), because ``solutions'' to ``systems of equations'' are computed by intersecting solutions to ``single equations''.  In order to define the operator $\CC$ we need to postulate that the functor $\I$ preserves all existing powers of $\a$ (Assumption \ref{ass:product-preserving} below), for this guarantees that $\CC$ ranges over $\a$-congruences.  The pair $(\CC,\VV)$ yields a Galois connection (Lemma \ref{lem:galois}) that satisfies an appropriate abstraction of the {\it Nullstellensatz}, as we show in Theorem \ref{thm:null}. Moreover, the Galois connection  lifts to an adjunction between $\D$ and $\Q$ (see Theorem \ref{thm:weakadj}), provided we assume  (Assumption \ref{ass:cosep} below) that $\a$ be an \emph{$\I$-coseparator}; see Definition \ref{def:coseparator}. 
\begin{table}[h!]
\adjustbox{max width=\columnwidth}{
\begin{tabular}{|p{4.8cm}|p{4.5cm}|p{3.5cm}|}
\hline
\textbf{Algebraic geometry} & \textbf{Universal algebra} & \textbf{Categories} \\ \hline
Ground field $k$                        & Any algebra $A$ in $\Va$                      & Functor $\I:\T \to \S $                \\ \hline
Class of $k$-algebras                    & Any variety $\Va$                      & Category $\Q$                 \\ \hline
$k[X_{1}, \ldots, X_{n}]$             & Free algebras $\mathcal{F}(\mu)$                         & Objects in $\T$                   \\ \hline
Polynomial  & Element of $\mathcal{F}(\mu)$ & $\T$-arrow $t\to\a$ \\ \hline
Ideal             & Congruence on of $\mathcal{F}(\mu)$                        & $\a$-congruence                  \\ \hline
Regular map                           & Term-definable map                        	& Definable map                    \\ \hline
Co-ordinate algebra of $S$       & Algebra presented by $\CC(S)$                    & Pair $(t, \CC{(S)})$ in $\Q$         \\ \hline
Affine variety                        & $\VV\circ\CC$-closed set             & Pair $(t,\VV{(R)})$ in $\D$                           \\ \hline 
\end{tabular}}
\vspace{0.3cm}
\caption{Corresponding concepts in the geometric, algebraic, and categorical settings.}\label{tab:transl}
\end{table}
This category-theoretic generalisation indicates that one can construct general affine adjunctions well beyond the standard setting of universal algebra investigated in Section \ref{sec:algebraic}; see Section \ref{syntacticcategories} for a discussion of possible applications in the context of syntactic categories and theories of presheaf type. 

In Section \ref{s:concreteness} we undertake the task of comparing our theory with the well-established theory of concrete dual adjunctions, through a number of constructions, results, and comments that help clarifying the relationship between the two theories. We postpone the relevant bibliographic references to that section. Here, on the other hand, we  mention some further literature that relates to our paper. 

In \cite{Diers} Diers develops a framework that abstracts the classical ring-theoretic adjunction recalled above to
 (possibly infinitary) varieties of algebras. For any given algebra $L$ in a  variety he establishes an adjunction between a 
category of ``affine subsets'' over $L$ and a category of ``algebraic systems'', as well as an adjunction between a category of 
``affine algebraic sets'' and the category of algebras of the given sort. This second adjunction restrict  to a duality between the former category and a 
category of ``functional algebras'' over $L$. Diers  does not focus on presented algebras, as we do with the intent of providing a direct abstraction of the notion of polynomial. Diers' affine (algebraic) sets amounts to  pairs $(X, A(X))$ consisting of a set $X$ and  a subalgebra $A(X)$ of the algebra $L^{X}$, and thus differ from ours. They also differ in another, more general respect, in that they  do not come  with an explicit embedding into an ambient space, as ours do.

In \cite{daniyarova2012algebraic} and in subsequent and related papers by these and other authors (see e.g.\ also \cite{Plotkin}), various elements of  algebraic geometry are developed for finitary varieties of algebras. Extensions of the theory to more expressive languages, such as Horn  or full first-order languages, are also considered.  The authors' main starting point is a circle of significant results in group theory (please see the introduction to \cite{daniyarova2012algebraic} for details). Their objective is to obtain a theory that affords the application and extension to other algebraic structures of the methods successfully applied in that group-theoretic context. The more general portions of the theory relate closely, at least in spirit, to what we try to do here. However, we are primarily concerned with gaining a deeper understanding of the basic adjunction at hand, and of its relation to the theory of concrete dual adjunctions. From this point of view, 
 \cite{daniyarova2012algebraic} and the related papers we know are uninformative. Theorem 5.6 in  \cite{daniyarova2012algebraic} does prove a dual equivalence between a category of ``algebraic sets over an algebra'' and a category of ``coordinate algebras'' for those algebraic sets; and, unlike the case of Diers' paper, these notions do coincide with ours in the case of finitary varieties. However, the authors' equivalence is not shown to descend  from a more general dual adjunction, nor is the adjoint to the functor inducing the dual equivalence described: the authors prove their theorem showing that the functor is fully faithful and essentially surjective. No mention of the theory of concrete dual adjunctions is made. Parts of our paper may thus be seen as a deepening and an extension of some of the more general results of   \cite{daniyarova2012algebraic}. 

A number of classical Stone-type dualities may be obtained by applying the theory developed in this paper. In Section \ref{p:dualities} we outline the method in the prototypical cases of Boolean algebras and commutative unital $C^{*}$-algebras.  A detailed proof of a duality theorem along the lines of the theory developed in  Section \ref{sec:algebraic} can be found in \cite{MarSpa12}, where semisimple MV-algebras are shown to be equivalent to closed subspaces of Tychonoff cubes with appropriate continuous maps. (On various aspects of the duality theory of  MV-algebras the interested reader can also see \cite{CaSp17,Marra-Reggio,MarSpa13}.)

Finally, we point out that the connection between general {\it Nullstellens\"atze}  and Birkhoff's subdirect representation theorem is addressed in \cite{tholen2013nullstellen}.


\section{The affine adjunction}\label{s:affadj}
 If $x$ and $y$ are objects in a category $\Cc$, and $f\colon x \to y$ is an arrow in $\Cc$, we write $\hom_{\Cc}{(x,y)}$ to denote the collection of arrows in $\Cc$ from $x$ to $y$, and $\dom{f}$ to denote the domain $x$ of $f$.

We consider the following:
\begin{itemize}
	\item Two categories $\T$ and $\S$, with $\S$ well-powered.
	\item A functor $\I \colon \T \to \S$.
	\item An object $\a$ of $\T$.
\end{itemize}
\subsection{The Galois connection given by the operators $\VV$ and $\CC$}\label{ss:CV}
For an object $t$,  by a \emph{$\a$-relation on $t$}, or simply  \emph{relation}, we mean a binary relation on the set $\hom_\T{(t,\a)}$ of $\T$-morphisms from $t$ to $\a$.
Along with relations, we consider subobjects in $\S$ defined by $\I$. That is, we consider pairs
$(t,s)$ where $t$ is a $\T$-object  and $s \colon \dom{s} \to\I(t)$ is an $\S$-subobject.  For a given $\T$-object $t$, we write
\[
\Sub{\I(t)}
\]
for the set of all $\S$-subobjects of $\I(t)$.
\begin{definition}\label{def:C}
	Given $s\in \Sub{\I(t)}$, we set
	\begin{align}\label{eq:RS}
	\CC{(s)}\coloneqq\left\{(p,q)\in \hom_{\T}^{2}{(t, \a)} \mid \I(p)\circ s=\I(q)\circ s\right\}.
	\end{align}
\end{definition}
To define a $\VV$-operator corresponding to $\CC$, we need:
\begin{assumption}\label{ass:limits} The category  $\S$ has equalisers of pairs of parallel arrows, and intersections of arbitrary families of subobjects. We denote the intersection of a family $\{E_i\}_{i\in I}$ of $\S$-subobjects by $\bigwedge_{i \in I} E_i$. 
\end{assumption}

\begin{definition}\label{def:V}
	For any relation $R$ on $\hom_{\T}{(t,\a)}$, with $t$ an object of $\T$, we set
	\begin{align}\label{eq:SR}
	\VV{(R)}:=\bigwedge_{(p,q)\in R}\Eq{(\I(p),\I(q))},
	\end{align}
	where,  for  $(p,q)\in R$,  $\Eq{(\I(p),\I(q))}$ denotes the $\S$-subobject of $\I(t)$ given by the equaliser  in $\S$ of the  $\S$-arrows $\I(p), \I(q) \colon \I(t) \rightrightarrows \I(\a)$.
\end{definition}
Writing $2^{E}$ for the power set  of the set $E$, we obtain functions (implicitly indexed by $t$)
\begin{align}
\CC&\colon \Sub{\I(t)}\longrightarrow 2^{\hom^2_{\T}{(t, \a)}}, \nonumber\\ 
\VV&\colon 2^{\hom^2_{\T}{(t, \a)}}\longrightarrow \Sub{\I(t)} \nonumber
\end{align}
that yield (contravariant) Galois connections. Here, the power set is ordered by inclusion, and $\Sub{\I(t)}$ is ordered by the usual partial order on subobjects,\footnote{Thus, explicitly, if $x$ and $y$ are subobjects of $z$, $x\leq y$ if there is an arrow $m\colon \dom{x} \to \dom{y}$ such that $x=y\circ m$.}
which we denote $\leq$.
\begin{lemma}[Galois connection]\label{lem:galois}
	For any  $\T$-object $t$, any relation $R$ on $\hom_\T{(t,\a)}$, and any $\S$-subobject $s \colon \dom{s}\to \I(t)$, we have
	\begin{align}\label{eq:galois}
	R \subseteq \CC{(s)} \quad \quad \text{if, and only if,} \quad \quad s\leq \VV{(R)}.
	\end{align}
\end{lemma}
\begin{proof}We have  $R \subseteq \CC(s)$ if, and only if, for any $(p,q) \in R$ it is the case that $\I{(p)}\circ s = \I{(q)}\circ s$. On the other hand, $s\leq \VV{(R)}$ if, and only if, there is an $\S$-arrow $m\colon \dom{s}\to \dom{\VV{(R)}}$ with $s=\VV{(R)}\circ m$.
	Now, if the former holds then $s$ must factor through $\VV{(R)}$ because the latter is defined in (\ref{eq:SR}) as the intersection of all $\S$-subobjects of $\I(t)$ that equalise some pair in $R$. Conversely, if the latter holds then for each $(p,q)\in R$ we obtain,  composing both sides of 
	$\I{(p)}\circ \VV{(R)} =\I{(q)}\circ\VV{(R)}$ with $m$, that  $\I{(p)}\circ s =\I{(q)}\circ s$.
\end{proof}

\subsection{From relations to congruences}\label{ss:FromRtoCong}
Consider again an object $t$ of $\T$. By a \emph{$\a$-congruence}  we mean an equivalence relation $R$ on $\hom_\T{(t,\a)}$ that satisfies the following property.
For any set $I$,  if the power $\a^{I}$ exists then for any morphism $g\colon\a^{I}\to \a$ and for any $I$-indexed family of pairs
\begin{align*}
P\coloneqq\big\{(p_{i},q_{i})\mid i \in I \text{ and } p_{i},q_{i}\in\hom_\T{(t,\a)}\big\},
\end{align*}
we have:
\begin{align}
\label{eq:abstract-cong} \text{If }P \seq R\text{ then } \left(g\circ \sprod{i\in I}p_{i},\ g\circ \sprod{i\in I}q_{i}\right)\in R.
\end{align}
In the Galois connection of Lemma \ref{lem:galois} the relation $\CC{(s)}$ always is an equivalence relation. However, $\CC{(s)}$ need not be a $\a$-congruence. To guarantee that, we need:
\begin{assumption}\label{ass:product-preserving}
	The functor $\I$ preserves all existing powers of $\a$. That is, if $I$ is a set and $\a^I$ exists in $\T$, then $\I(\a^I)\cong\I(\a)^I$.
\end{assumption}

\begin{remark}Using equivalence relations in place of $\a$-congruences one can establish a weaker variant of the adjunction  in Theorem \ref{thm:weakadj} below. For such a variant Assumption \ref{ass:product-preserving} is not needed. Nonetheless, the assumption is convenient for the subsequent specialisation of Theorem \ref{thm:weakadj} to the algebraic setting, and instrumental in establishing the representability property expressed by Proposition \ref{l:Rrepresentable}. 	
\end{remark}

\begin{lemma}\label{l:C-congruence}
	For any $\T$-object $t$ and any $s\in \Sub{\I(t)}$, the set $\CC{(s)}$ defined in \eqref{eq:RS} is a $\a$-congruence  on $\hom_{\T}{(t, \a)}$.
\end{lemma}
\begin{proof}
	It is straightforward to check that $\CC{(s)}$ is an equivalence relation, so we only show that \eqref{eq:abstract-cong} holds.  Suppose the power $\a^{I}$ exists in $\T$ for some set $I$, and let $(p_{i},q_{i})\in \CC{(s)}$  for $i\in I$.  Applying $\I$ to the product arrow $\Pi_{i\in I}p_i$, we obtain the following diagram.
	
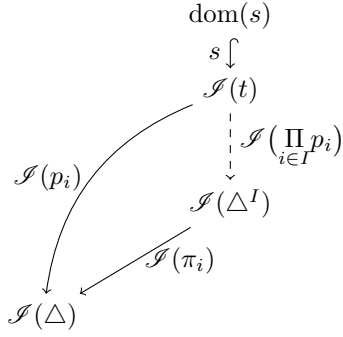
\begin{center}
		\begin{tikzpicture}[scale=0.5]
		\node (a) at (5,8)   {$\dom(s)$};
		\node (c) at (5,6)   {$\I(t)$};
		\node (Dn) at (5,3)   {$\I(\a^{I})$};
		\node (d1) at (0,0) {$\I(\a)$};
		\draw [right hook->] (a) to node [midway,left] {$s$} (c);
		\draw [->] (Dn) -- (d1) node [midway,right] {$\I(\pi_{i})$};
		\draw [dashed,->] (c) to node [midway,right] {$\I\big(\sprod{i\in I}p_{i}\big)$} (Dn);
		\draw [->] (c) to [bend right] node[left, midway] {$\I(p_{i})$} (d1);
		\end{tikzpicture}
	\end{center}
	By Assumption \ref{ass:product-preserving} we have $\I(\a^{I})=\I(\a)^{I}$.  Hence, all arrows $\I(p_{\beta})\circ s$ must factor through $\I(\a^{I})$ and $\I\big(\Pi_{i\in I}p_{i}\big)\circ s$ is the unique arrow that factors them all.  There is a corresponding diagram  that uses the arrows $q_{i}$; the same argument yields that all arrows $\I(q_{i})\circ s$ factor through $\I\big(\Pi_{i\in I}q_{i}\big)\circ s$.  But $(p_{i},q_{i})\in \CC{(s)}$, and  by Definition \ref{def:C} we have  $\I(p_{i})\circ s=\I(q_{i})\circ s$, so that also their factorisations must be equal.  We conclude  
	\begin{align*}
	\I\big(\sprod{i \in I}p_{i}\big)\circ s=\I\big(\sprod{i\in I}q_{i}\big)\circ s\quad\text{ for each } i \in I.
	\end{align*}
	If $g\colon\a^{\alpha}\to \a$ is any  $\T$-arrow, 
	\begin{align*}
	\I\big(g\circ \sprod{i\in I}p_{i}\big)&\circ s=&\\
	&\I(g)\circ \I\big(\sprod{i \in I}p_{i}\big)\circ s=\I(g)\circ \I\big(\sprod{i \in I}q_{i}\big)\circ s&\\
	&&=\I\big(g\circ\sprod{i \in I}q_{i}\big)\circ s 
	\end{align*}
	for each $i \in I$.
	In light of Definition \ref{def:C} this entails  $\big(g\circ \Pi_{i \in I}p_{i}, g\circ\Pi_{i \in I}q_{i}\big)\in \CC{(s)}$, and the lemma is proved. 
\end{proof}
\subsection{The category $\Q$ of quotients modulo $\a$-congruence relations}\label{ss:R}
We next define the category $\Q$. Objects in $\Q$ are quotient sets 
\[
\frac{\hom_{\T}{(t, \a )}}{R},
\]
where $t$ is a $\T$-object and $R$ is a $\a$-congruence on $\hom_{\T}{(t, \a )}$. Sometime, we more succinctly indicate objects in $\Q$ as pairs $(t,R)$.  To define morphisms, let us first note 
that arrows in $\T$ may or may not preserve relations, in the following sense. Let  $f\colon t\to t'$ be a  $\T$-arrow. Given $\a$-relations $R$ and $R'$ on $t$ and $t'$, respectively, if the function 
\begin{align*}
-\circ f\colon  \hom_{\T}{(t', \a)} \to \hom_{\T}{(t, \a)}
\end{align*}
satisfies the property
\begin{align}\label{eq:Rarrow}
(p', q')\in R' \quad \Longrightarrow \quad (p' \circ f, q' \circ f)\in R,
\end{align}
we say that $f$ \emph{preserves $R'$} (\emph{with respect to $R$}). By definition, there is an arrow 
\[
(t,R) \overset{g^{\rm op}}{\longrightarrow} (t',R')
\]
in $\Q^{\rm op}$ precisely when there is a function
\[
\frac{\hom_{\T}{(t', \a )}}{R'}\overset{g}{\longrightarrow} \frac{\hom_{\T}{(t, \a )}}{R}
\] 
such that there exists an $R'$-preserving $\T$-arrow $f\colon t\to t'$ making the diagram
\begin{small}
	\[
	\begin{tikzpicture}[scale=0.6]
	\node (hom1R) at (0,0)   {$\hom_{\T}{(t', \a )}/R'$};
	\node (hom2R) at (5,0) {$\hom_{\T}{(t, \a )}/R$};
	\node (hom1) at (0,3) {$\hom_{\T}{(t',\a)}$};
	\node (hom2) at (5,3) {$\hom_{\T}{(t,\a)}$};
	\draw [->] (hom1) -- (hom2) node [above, midway] {$-\circ f$};
	\draw [->>] (hom1) -- (hom1R);
	\draw [->>] (hom2) -- (hom2R);
	\draw [->] (hom1R) -- (hom2R) node [below, midway] {$g$};
	\end{tikzpicture}
	\]
\end{small}commute, where the vertical arrows are the natural quotient functions. We then say that $f$ \emph{induces} $g^{\rm op}$ (and $g$).

\subsection{The category $\D$ of subobjects and definable morphisms}\label{ss:D}
Objects of $\D$ are all pairs $(t,s)$ where $t$ is $\T$-object  and $s \colon \dom{s} \to\I(t)$ is an $\S$-subobject.   If  $(t,s),(t',s')$ are objects in $\D$ and $g \colon \dom{s} \to \dom{s'}$ is an $\S$-arrow, then $g$ is a $\D$-arrow from  $(t,s)$ to $(t',s')$ if there exists a $\T$-arrow $f \colon t \to t'$ such that  the diagram below
\begin{small}
	\[
	\begin{tikzpicture}[scale=0.4]
	\node (S) at (0,0)   {$\dom{s}$};
	\node (VTS) at (0,5) {$\I(t)$};
	\node (T) at (5,0) {$\dom{s'}$};
	\node (VTT) at (5,5) {$\I(t')$};
	\draw [->] (S) -- (T) node [below, midway] {$g$};
	\draw [->] (VTS) -- (VTT) node [above, midway] {$\I(f)$};
	\draw [->] (S) -- (VTS) node [left, midway] {$s$};
	\draw [->] (T) -- (VTT) node [right, midway] {$s'$};
	\end{tikzpicture}
	\]
\end{small}
commutes.  We then say that $f$ \emph{induces} $g$.

\subsection{The functor  $\C\colon \D\to\Q^{\rm op}$} For any $\D$-object $(t, s)$, we set
\begin{align*}
\C(t, s):=\frac{\hom_{\T}{(t,\a)}}{\CC{(s)}}.
\end{align*}
If $g\colon(t, s)\to (t', s')$  is  a $\D$-arrow induced by $f\colon t\to t'$, we let 
\[
\C(f)\colon \frac{\hom_{\T}{(t',\a)}}{\CC{(s')}}\longrightarrow\frac{\hom_{\T}{(t,\a)}}{\CC{(s)}}
\]
be the $\Q$-arrow induced by $f$. 
\begin{lemma}
	The functor $\C$ is well-defined.
\end{lemma}
\begin{proof}
	We keep the notation used above in the definition of $\C$.  We first check that $f$ preserves $\CC(s')$ with respect to $\CC(s)$, i.e., the function
	\begin{align*}
	-\circ f\colon \hom_{\T}{(t', \a )} \to \hom_{\T}{(t, \a )} 
	\end{align*}
	satisfies $(p' \circ f, q' \circ f)\in \CC(s)$ for any $(p', q')\in \CC(s')$ as required by \eqref{eq:Rarrow}. Indeed,
	\begin{align}\label{eq:welldef1}
	\I(f)\circ s= s'\circ g \ ,
	\end{align}
	because  $g$ is induced by $f$.  Now,  given $p', q'\in \hom_{\T}{(t', \a )}$, assume $(p',q')\in \CC(s')$, that is 
	\begin{align}\label{eq:welldef2}
	\I(p')\circ s'=\I(q')\circ s'. 
	\end{align}
	Composing both sides of (\ref{eq:welldef2}) with $g$, and applying (\ref{eq:welldef1}), we obtain $\I(p'\circ f)\circ s = \I(q'\circ f)\circ s$, which shows $(p'\circ f,q'\circ f)\in \CC(s)$.
	Finally, the definition does not depend on the choice of the arrow inducing $g$.  Indeed, if $f$ and $f'$ induce the same $\D$-arrow $g$, then 
	\begin{equation}
	\label{eq:1}\I(f)\circ s=\I(f')\circ s\ .
	\end{equation}
	We need to show that the factorisations of $-\circ f, -\circ f'\colon \hom_{\T}{(t', \a )} \rightrightarrows \hom_{\T}{(t, \a )}$ through the quotient sets $\hom_{\T}{(t', \a )}/\CC(s')$ and $\hom_{\T}{(t, \a )}/\CC(s)$ are equal. That is, for all $\f\in\hom{(t',\a)}$ the equality $(\f\circ f)/\CC(s)=(\f\circ f')/\CC(s)$ holds, or equivalently, $((\f\circ f),(\f\circ f'))\in \CC(s)$.  The latter means by definition, cf.\ (\ref{eq:RS}), that $\I(\f\circ f)\circ s=\I(\f\circ f')\circ s$, which can be obtained from (\ref{eq:1}) above by composing both sides with $\I(\f)$.
\end{proof}

\subsection{The functor $\V\colon \Q^{\rm op} \to \D$} 
For any $\Q$-object $(t,R)$, we set
\begin{align*}
\V(t,R):=(t,\VV{(R)}).
\end{align*}
For a $\Q$-arrow  $g^{\rm op}\colon \hom_{\T}{(t',\a)}/{R'}\to \hom_{\T}{(t,\a)}/{R}$ induced by a $\T$-arrow $f\colon t\to t'$, we define $\V(g^{\rm op})$ to be the $\D$-arrow induced by $f$.

For the definition of $\V$ to make sense an assumption on $\a$ is needed. Let us  introduce a generalisation of the classical notion of coseparator.
\begin{definition}\label{def:coseparator}
	Given a functor $F\colon\A\to\B$ and an $\A$-object $x_{1}$ of $\A$, we say that $x_{1}$ is an \emph{F-coseparator} if for any $\A$-object $x_{2}$ and any pair of distinct $\B$-arrows $g_{1}\neq g_{2}\colon y\to F(x_{2})$ there is an $\A$-arrow $f\colon x_{2}\to x_{1}$ such that
	\[
	F(f)\circ g_{1}\neq F(f)\circ g_{2}.
	\]
\end{definition}
Observe that if $1_{\A}\colon \A\to\A$ is the identity functor, a $1_{\A}$-coseparator is just a coseparator in $\A$.
\begin{lemma}
	If $\a$ is an $\I$-coseparator, then the functor $\V$ is well-defined.
\end{lemma}
\begin{proof}
	We first need to show that $\I(f)\circ \VV{(R)}$ factors through $\VV{(R')}$. Indeed, let $p',q' \in \hom_{\T}{(t',\a)}$, and assume $(p',q') \in R'$. Then $(p' \circ f, q'\circ f)\in R$ because $f$ induces a $\Q$-arrow, and therefore $\I(p')\circ (\I(f) \circ \VV{(R)})=\I(q')\circ (\I(f)\circ \VV{(R)})$ for all $(p',q')\in R'$. By the universal property of the pullback $\VV(R')\coloneqq\bigwedge_{(p',q')\in R}\Eq{(\I(p'),\I(q'))}$ it follows that $\I(f)\circ \VV{(R)}$ factors through $\VV(R')$.
	
	To see that the definition does not depend on the choice of the arrow inducing $g^{\rm op}$ suppose that $f$ and $f'$ induce the same $g\colon (t, R)\to (t', R')$. This happens if, and only if, for all $\f\colon t'\to \a $ we have $(\f\circ f,\f\circ f')\in R$, which in turn entails $\I(\f\circ f)\circ \VV{(R)}=\I(\f\circ f')\circ \VV{(R)}$, by the definition (\ref{eq:SR}) of $\VV{(R)}$ as an intersection of equalisers.  Since $\a$ is assumed to be an $\I$-coseparator we conclude that $\I(f)\circ \VV{(R)}=\I(f')\circ \VV{(R)}$.  But $\VV{(R')}\circ \V(f)= \I(f)\circ \VV(R)$ and $\VV{(R')}\circ \V(f')= \I(f')\circ \VV(R)$ so, $\VV{(R')}\circ \V(f)=\VV{(R')}\circ \V(f')$ and since $\VV(R')$ is monic we finally obtain $\V(f)=\V(f')$.
\end{proof}
\begin{assumption}\label{ass:cosep}We henceforth assume that the $\T$-object $\a$ is an $\I$-coseparator.
\end{assumption}
\subsection{The dual adjunction}\label{ss:dualadj}\label{subs:weak}
The Galois connection (\ref{eq:galois}) lifts to an adjunction.

\begin{theorem}[Affine adjunction]\label{thm:weakadj}
	The functor $\C \colon \D\to\Q^{\rm op}$ is left adjoint to the functor $\V\colon \Q^{\rm op}\to\D$. In symbols, $\C\dashv \V$.
\end{theorem}   

\begin{proof}
	Let us show that for any $\D$-object $(t, s)$  and any $\Q^{\rm op}$-object $(t', R')$ we have a natural bijective correspondence between the $\Q^{\rm op}$-arrows $\C(t, s)=(t, \CC(s))\to (t', R')$ and the $\D$-arrows $(t, s)\to (t', \VV(R'))=\V(t', R')$.
	
	On the one hand, $(t, \CC(s))\to (t', R')$ is a $\Q^{\rm op}$-arrow if, and only if, it is induced by a $\T$-arrow $f\colon t\to t'$ such that for any $p', q' \in \hom_{\T}{(t', \a)}$, if $(p', q')\in R'$ then $\I(p') \circ \I(f)\circ s=\I(q') \circ \I(f) \circ s$. On the other hand, $f$ defines a $\D$-arrow $(t, s)\to (t', \VV(R'))$ in $\D$ if, and only if, $f\circ s$ factors through $\VV(R')$, i.e.\ for any $p', q' \in \hom_{\T}{(t', \a )}$, if $(p', q')\in R'$ then $\I(p') \circ \I(f)\circ s=\I(q') \circ \I(f) \circ s$. It is thereby clear that $\C\dashv\V$.
\end{proof}

\subsection{A  {\em Nullstellensatz} for the affine adjunction}\label{ss:weaknull}
Recall that a collection of arrows $A\subseteq \hom_\A{(x,y)}$ in a category $\A$ is \emph{jointly epic} if whenever $f_1,f_2\colon y\rightrightarrows z$ are $\A$-arrows with $f_1\circ g = f_2\circ g$ for all $g\in A$, then $f_1=f_2$.
\begin{theorem}\label{thm:null}
Fix a $\Q$-object $(t,R)$. For any family $\Sigma=\{\sigma_{i}\}_{i\in I}$ of subobjects of $\I{(t)}$ such that for each $\sigma_{i}$ there exists $m_{i}$ with $\sigma_{i}=\VV{(R)}\circ m_{i}$ \textup{(}i.e.\ $\sigma_{i}\leq \VV{(R)}$\textup{)} and the family of $\S$-arrows $\{m_{i}\}_{i\in I}$ is jointly epic in $\S$, the following are equivalent.
\begin{enumerate}[\textup{(}i\textup{)}]
\item $R=\CC{(\VV{(R)})}$, i.e.\ $R$ is fixed by the Galois connection \textup{(\ref{eq:galois})}.
\item $R=\bigcap_{i\in I}\CC(\sigma_{i})$.
\end{enumerate}
\end{theorem}
\begin{proof}
First observe that the Galois connection (\ref{eq:galois}) in Lemma \ref{lem:galois} implies the \emph{expansiveness} of $\CC\circ\VV$, i.e.
\begin{align}\label{eq:contained1}
R\seq \CC{(\VV{(R)})}\ .
\end{align}
Further, since each $\sigma_{i}\leq \VV{(R)}$, again by general properties of Galois connections, it follow that
\begin{align}\label{eq:contained2}
R\seq \bigcap_{i\in I}\CC(\sigma_{i})\ .
\end{align}

\smallskip \noindent (i)$\Rightarrow$(ii)\, 
As by hypothesis $R=\CC{(\VV{(R)})}$, by (\ref{eq:contained2}) above, it is enough to prove 
\[\bigcap_{i\in I}\CC(\sigma_{i})\seq \CC{(\VV{(R)})}\ .\] 
  If $(p,q)\in \bigcap_{i\in I}\CC(\sigma_{i})$, then for every $\sigma_{i}\in \Sigma$, $\I(p)\circ \sigma_{i}=\I(q)\circ \sigma_{i}$.  By hypothesis, the latter can be rewritten as $\I(p)\circ \VV{(R)}\circ m_{i}=\I(q)\circ \VV{(R)}\circ m_{i}$. Now, the family of factorisations $\{m_{i}\}_{i\in I}$ is jointly epic in $\S$, hence we obtain $\I(p)\circ \VV{(R)}=\I(q)\circ \VV{(R)}$, which proves $(p,q)\in \CC{(\VV{(R)})}$.

\smallskip \noindent (ii)$\Rightarrow$(i)\,
By (\ref{eq:contained1}) above and the hypothesis (ii), it is enough to prove 
\[\CC{(\VV{(R)})}\subseteq \bigcap_{i\in I}\CC(\sigma_{i})  .\] 
Suppose that $(p,q)\in\CC{(\VV{(R)})}$, i.e.\  $\I(p)\circ \VV{(R)}=\I(q)\circ \VV{(R)}$. By composing on the right with $m_i$ we obtain,
for all $\sigma_{i}\in \Sigma$, $\I(p)\circ \VV{(R)}\circ m_{i}=\I(q)\circ \VV{(R)}\circ m_{i}$.  Applying the above commutativity of $\sigma_{i}$ one obtains $\I(p)\circ \sigma_{i}=\I(q)\circ\sigma_{i}$.  The latter entails that, for all $i\in I$, $(p,q)\in \CC{(\sigma_{i})}$, whence $(p,q)\in \bigcap_{i\in I}\CC(\sigma_{i})$. 
\end{proof}

\begin{remark}
Notice that a   family $\Sigma$ satisfying the hypotheses of Theorem \ref{thm:null} always exists, namely the singleton $\dom{\VV{(R)}}$ with the arrow $\VV{(R)}$.  In this case the theorem has no actual content.  By contrast, when the category $S$ is $\Set$,  one can chose as $\Sigma$ the family of maps with domain the singleton $\{*\}$.  The family $\Sigma$ is obviously jointly surjective, and Theorem \ref{thm:null} can be restated in a more concrete form as follows.
\end{remark}
\begin{theorem}\label{thm:null-in-set}
Suppose $\S=\Set$.  For any $\Q$-object $(t,R)$ the following are equivalent.
\begin{enumerate}[\textup{(}i\textup{)}]
\item $R=\CC{(\VV{(R)})}$,
\item $R=\bigcap_{\sigma\leq \VV{(R)}}\CC(\sigma)$, where $\sigma$ ranges over all $\S$-subobjects $\{*\}\to \I(t)$.
\end{enumerate}
\end{theorem}

\section{Concreteness and representability}\label{s:concreteness}We would like to compare the adjunction of Theorem \ref{thm:weakadj} to the theory of dual adjunctions induced by a dualising object. The literature on the latter topic is considerable. The monographs \cite{johnstone, clark1998natural} contain extensive bibliographies. The general theory is concisely but comprehensively developed in \cite{DT89} and \cite{PT91}. It will be convenient for our purposes to use \cite{PT91} as main reference in this section.

\subsection{Concreteness of $\Q$}\label{ss:concreteM}
There is an obvious faithful functor
\begin{align}
\label{eq:Rforget}
U_{\Q} \colon \Q\longrightarrow \Set,
\end{align}
namely,  the inclusion functor of $\Q$ into $\Set$. Thus, $\Q$ comes equipped with the structure of a concrete category.
\begin{proposition}\label{l:Rrepresentable}The faithful functor $U_{\Q} \colon \Q\longrightarrow \Set$ in \eqref{eq:Rforget} is represented by the object \[\frac{\hom_{\T}{(\a,\a)}}{\id_{\a}}\] of $\Q$, where $\id_{\a}$ denotes the identity relation on the set $\hom_{\T}{(\a,\a)}$.
\end{proposition}
\begin{proof}
	We need to provide a set-theoretical bijection $$\hom_{\Q}((\a, \id_{\a}), (t, R))\cong \hom_{\T}(t, \a)\slash R$$ naturally in $(t, R)\in \Q$. By definition, the arrows $(\a, \id_{\a}) \to (t, R)$ in $\Q$ are the arrows $f\colon t \to \a$ in $\T$ (as all of them preserve $\id_{\a}$), modulo the equivalence relation $\rho$. Recall that $f\rho f'$ if, and only if, the factorisations of $-\circ f, -\circ f'$  through $\hom_{\T}(\a, \a)\slash \id_{\a}$ and $\hom_{\T}(t, \a)\slash R$ (which by an abuse of notation we still indicate by $-\circ f$ and $-\circ f'$) are equal. This latter condition is equivalent to saying that $(f, f')\in R$. Indeed, if $-\circ f= -\circ f'$ then $(-\circ f)(1_{\a})=(-\circ f')(1_{\a})$ in $\hom_{\T}(t, \a)\slash R$, i.e.\ $(f, f')\in R$. For the other implication, notice that, by assumption, $R$ is a $\a$-congruence. This implies that, for any arrow $g\in\hom_{\T}(\a,\a)$, $(g\circ f,g\circ f')\in R$, hence $(-\circ f)(g)$ is equal to $(-\circ f')(g)$ in $\hom_{\T}(t, \a)\slash R$. This proves that sending a $\Q$-arrow from $(\a,\id_{\a})$ to $(t,R)$ into its $R$-equivalence class in $\hom_{\T}(t,\a)$ gives a bijection, which is clearly natural.
\end{proof}

\subsection{(Lack of) concreteness of $\D$}\label{ss:notconcrete} Contrary to what we just saw for $\Q$, it is not clear how to equip $\D$ with the structure of a concrete category. On the other hand, more can be said if we restrict to the case\footnote{We mention that, {\it mutatis mutandis}, we could develop some of the  considerations that follow  under the more general assumption that $\S$ has itself the structure of a concrete category.} $\S=\Set$. There is then a  forgetful functor
\begin{align}
\label{eq:Sforget}
U_{\D} \colon \D\longrightarrow \Set
\end{align}
that sends the object $(t,s)$ of $\D$ to $\dom(s)$, and acts on  arrows accordingly. Although we do not have a complete characterisation of when $U_{\D}$ is representable, nor do we know whether one is feasible, we can at least relate the representability of $U_{\D}$ to that of the functor $\I$, as follows.
\begin{proposition}\label{p:repres}
	The following conditions are equivalent:
	\begin{enumerate}[\textup{(}i\textup{)}]
		\item The functor $\I$ is representable.
		
		\item The functor $U_{\D}$ in \eqref{eq:Sforget} is representable by an object of the form $(\b, 1)$ where $1\colon\{\ast\}\to \I(\b)$ and the element $1(\ast)\in \I(\b)$ satisfies the property that the arrows $f\colon\b \to t$ in $\T$ are in bijection with the elements of $\I(t)$ via the assignment $f\mapsto \I(f)(1(*))$.
	\end{enumerate}
\end{proposition}  

\begin{proof}
	Suppose that the functor $\I$ is representable and that $\b$ is the object representing it. Let us prove that the functor $U_{\D}$ is represented by the object $(\b, 1)$, where $1$ is the subobject $\{\ast\} \to \I(\b)\cong \hom_{\T}(\b, \b)$ picking out the identity on $\b$. Clearly, the element $1(\ast)$ will satisfy the condition in (ii) by definition of a representable functor.
	
	We have to show that that there is a natural bijective correspondence between the arrows $(\b, 1) \to (t, s)$ in the category $\D$ and the elements of $\dom(s)$. For this, it suffices to recall that giving an arrow $(\b, 1) \to (t, s)$ in $\D$ consists precisely in giving an arrow $\dom(1)=\{\ast\}\to \dom(s)$, that is an element $x$ of $\dom(s)$, induced by an arrow $t\colon\b \to t$ in $\T$. But this latter condition is satisfied if we take $t$ equal to the element of $\I(t)\cong \hom_{\T}(\b, t)$ given by the image of $x$ under $s$, whence the arrows $(\b, 1) \to (t, s)$ in $\D$ correspond precisely to the elements of $\dom(s)$.  
	
	Conversely, suppose that $(\b, 1)$, where $1\colon\{\ast\}\to \I(\b)$, is a representing object for $U_{\D}$ such that the element $1(\ast)$ satisfies the property in (ii). Then for any object $t$ of $\T$, the map which assigns to any arrow $f:\b \to t$ in $\T$ the element $\I(f)(1(\ast))$ is a bijection from $\hom_{\T}(\b, t)$ onto the set of elements of $\I(t)$. This implies in particular that $\I$ is represented by the object $\b$.
\end{proof}

\begin{remark}
	The condition that the functor $U_{\D}$ be represented by an object $(\b, 1)$ satisfying the requirements in condition (ii) of Proposition \ref{p:repres} is notably met in a number of classical dualities, e.g.\ Stone duality for Boolean algebras or Stone-Gelfand duality for commutative unital C$^*$-algebras (see Section \ref{p:dualities}). Let us emphasise, moreover, that the condition in (ii) on the element $1(\ast)$ is automatically satisfied if $\I(\b)$ is a singleton.\end{remark}

The point of Proposition \ref{p:repres} is that whenever the functor $U_{\D}$ is representable and faithful then, in view of Proposition \ref{l:Rrepresentable}, the theory of concrete dualities \cite{PT91, DT89} applies to the adjunction $\C\dashv \V$ of Theorem \ref{thm:weakadj}, providing a wealth of information. Thus (see \cite[Section 1-B]{PT91}), the contravariant functors $U_{\D}\circ \V\colon \Q\to \Set$ and $U_{\Q}\circ \C\colon \D\to \Set$ are represented by the objects 
\begin{align*}
a& \coloneqq \C{(\b)} \text{ and}\\
b &\coloneqq\V(\,(\a,\id_{\a})\,)
\end{align*}
of $\Q$ and $\D$, respectively. Further,  there is a bijection of the underlying sets
\begin{align}\label{eq:dualisingbij}
\beta\colon U_{\D}(a)\cong U_{\Q}(b).
\end{align}
Therefore $a$ and $b$ may be conceived of as a single set equipped with both the structure of a $\Q$-object and with the structure of a $\D$-object. Up to the bijection \eqref{eq:dualisingbij}, moreover, the components of the unit and co-unit of   the  adjunction  are necessarily given ``by evaluation'' in the precise sense of \cite[p.\ 115]{PT91}. The triple $(a,\beta,b)$ is  a \emph{dualising} or \emph{schizophrenic object} for the concrete categories $\Q$ and $\D$. As far as we see, the concrete  adjunction $\C\dashv \V$ induced by $(a,\beta,b)$ need not be natural in the sense of \cite[Section 1-C]{PT91}, meaning that certain initial lifts may not exist  --  cf.\ conditions (SO1) and (SO2) of \cite[p.\ 116]{PT91} for details.

\subsection{Representable extensions of $\I$}\label{ss:yoneda} While the representability of the functor $\I$ has the important consequence of inducing the structure of a concrete category on $\D$ through the functor $U_{\D}$ of Proposition \ref{p:repres}, in general $\I$ will fail to be representable. We will nonetheless see in Remark \ref{extendedrepresentability} that $\I$ always admits an appropriate extension to a representable functor in the context of varieties. Let us then consider the functor
\begin{align*}
\hom_{\T}{(-,\a)}\colon \T^{\rm op}\longrightarrow \Set.
\end{align*}
In order to compare this to the functor $U_{\Q}$ in \eqref{eq:Rforget}, and motivated by the well-understood case of algebraic theories, we define the functor
\begin{align*}
I_{\a}^{\rm op}\colon\T^{\rm op}\longrightarrow \Q
\end{align*}
that sends an object $t$ to the $\Q$-object
\[
\frac{\hom_{\T}(t, \a)}{\id_{t}},
\]
where $\id_{t}$ is the identity congruence relation on $\hom_{\T}(t, \a)$. Concerning  arrows, note that any $\T$-arrow $f\colon t \to t'$  necessarily preserves the identity congruence $\id_{t}$, and induces precisely one  $\Q$-arrow $g_{f}\colon (t',R)\to (t,R)$, again because  $\id_{t}$ is the identity relation. We set
$I_{\a}^{\rm op}(f)\coloneqq g_{f}$. Observe moreover that $I_{\a}^{\rm op}$ is always full, because  $\Q$-arrows must be induced by $\T$-arrows, and it is faithful if and only if $\a$ is a coseparator in $\T$.

Note that, by definition of the functors $U_{\Q}$ and $I_{\a}^{\rm op}$, the diagram
	\[
	\begin{tikzpicture}[scale=0.45]
	\node (T) at (0,5) {$\T^{\rm op}$};
	\node (A) at (6,2) {$\Q$};
	\node (S) at (6,5) {$\Set$};
	\draw [->] (T) -- (A) node [below, midway] {{\footnotesize $I_{\a}^{\rm op}$}};
	\draw [->] (T) -- (S) node [above, midway] {{\footnotesize $\hom_{\T}{(-,\a)}$}};
	\draw [->] (A) -- (S) node [right, midway] {{\footnotesize $U_{\Q}$}};
	\end{tikzpicture}
	\]
commutes.

Now, changing the variance, we consider the functor
\[
I_{\a}\colon \T\longrightarrow \Q^{\rm op}
\]
determined by $I_{\a}^{\rm op}$ in the obvious manner.
In those cases when  $\I\colon\T\to\Set$ fails to be representable, but does extend to a representable functor $\I'$ on $\Q^{\rm op}$ along the functor $I_{\a}$, we can  guarantee that the representable extension $\I'$ still satisfies our basic Assumptions \ref{ass:limits}--\ref{ass:cosep}. 

\begin{proposition}\label{p:restriction}
	Suppose that $\I'\colon\Q^{\rm op} \to \S$ is a functor such that $\I'\circ I_{\a}=\I$. If $\I$ and $\a$ satisfy Assumptions \ref{ass:limits}, \ref{ass:product-preserving}, and \ref{ass:cosep}, then so do $\I'$ and $\a'\coloneqq I_{\a}(\a)$.
\end{proposition}

\begin{proof}
We start noticing that for any object $(t, R)$ of $\Q$, we have a monic arrow $i_{(t, R)}\colon(t, R) \to \I_{\a}(t)$ in $\Q^{\rm op}$ which corresponds to the canonical projection $\hom_{\T}{(t,\a)} \to \hom_{\T}{(t,\a)}\slash R$, and for any arrow $f\colon(t, R)\to (t', R')$ in $\Q^{\rm op}$, we have that $i_{(t', R')}\circ f=I_{\a}(f)\circ i_{(t, R)}$ in $\Q^{\rm op}$. 

Since the category $\S$ does not change in passing from $\I$ and $\a$ to $\I'$ and $\a'$, Assumption $1$ is satisfied in the latter setup if it is in the former.

The fact that $\I'$ preserves all existing powers of $\a':=\I_{\a}(\a)$ follows from the fact that $\I'\circ {I_{\a}}=\I$ noticing that, since the functor ${I_{\a}}$ is right adjoint to the canonical projection functor $\Q^{\rm op} \to \T$, as  is easily verified, it necessarily preserves all existing limits. So Assumption $2$ is satisfied by $\I'$ and $\a'$ if it is by $\I$ and $\a$. 

Concerning Assumption $3$, we observe that, since the functor ${I_{\a}}$ preserves all existing limits in  $\Q^{\rm op}$, it sends in particular monic arrows to monic arrows. So any arrow of the form $\I'(i_{(t, R)}):\I'(t, R)\to \I'(I_{\a}(\a))=\I(\a)$ is a monomorphism in $\S$. It then follows that, given two arrows $\alpha, \beta:S \to \I'(t, R)$ in $\S$, $\alpha=\beta$ if and only of $\I'(i_{(t, R)})\circ \alpha=\I'(i_{(t, R)})\circ \beta$. Supposing that Assumption $3$ holds for $\I$ and $\a$, we want to prove that if $\I'(f)\circ \alpha=\I'(f)\circ \beta$ for every arrow $f:(t, R)\to \a'$ in $\Q^{\rm op}$ then $\alpha=\beta$. Now, we have that $f=i_{I_{\a}(\a)}\circ f= I_{\a}(f)\circ i_{(t, R)}$, whence $\I'(f)=\I'(I_{\a}(f))\circ \I'(i_{(t, R)})=\I(f)\circ \I'(i_{(t, R)})$. So the condition 
\[\I'(f)\circ \alpha=\I'(f)\circ \beta\text{ for every arrow }f:(t, R)\to \a'\text{ in }\Q^{\rm op}\]
 is equivalent to the condition 
 \[\I(f)\circ (\I'(i_{(t, R)}) \circ \alpha)=\I(f)\circ (\I'(i_{(t, R)}) \circ \beta)\text{ for every arrow }f:t \to \a\text{ in }\T,\] 
 which proves our thesis.       	
\end{proof}

\begin{proposition}\label{r:restriction}
	If the functor ${I_{\a}}\colon\T \to {\Q}^{\rm op}$ is \textup{(}full and\textup{)} faithful \textup{(}equivalently, if $\a$ is a coseparator in $\T$\textup{)}, then the adjunction $\C\dashv \V$ of Theorem \ref{thm:weakadj} induced by $\I$ and $\a$ is the restriction of the adjunction $\C'\dashv \V'$, where $\C' \colon \D'\to\Q'^{\rm op}$ and $\V'\colon \Q'^{\rm op}\to\D'$, induced by $\I'$ and $\a'$ along the full embeddings of $\D$ into $\D'$ and of $\Q^{\rm op}$ into $\Q'^{\rm op}$ induced by $I_{\a}$.  
\end{proposition}

\begin{proof}
The functor $I_{\a}$ being full and faithful, we have full embeddings $i:\D \to \D'$ and $j:\Q^{\rm op}\to \Q'^{\rm op}$ sending respectively any object $(t, s)$ of $\D$ to the object $(I_{\a}(t), s)$ of $\D'$ (notice that $s$ can be regarded as a subobject of $\I'(I_{\a}(t))$ since $\I'\circ I_{\a}=\I$) and any object $(t, R)$ of $\Q$ to the object $(I_{\a}(t), R')$ of $\Q'$, where $R'$ is the equivalence relation on the set $\hom_{\Q'^{\rm op}}(I_{\a}(t), \a')$ corresponding to $R$ under the bijection $\hom_{\Q'^{\rm op}}(I_{\a}(t), \a')\cong \hom_{\T}(t, \a)$ induced by $I_{\a}$. The verifications that $\C' \circ i=j\circ \C$, $\V' \circ J=i\circ \V$ and that the adjunction $\C'\dashv \V'$ restricts to the adjunction $\C\dashv \V$ along $i$ and $j$ present no difficulties.     	
\end{proof}

%
\section{Varieties of algebras}\label{sec:algebraic}
We henceforth restrict attention to algebraic categories.  In particular,  we shall work with \emph{varieties of algebras} (i.e., equationally definable classes of algebras) in the sense of S{\l}ominsky \cite{Slominsky:1959} and Linton \cite{Linton:1965}. Notation:%
\begin{itemize}
\item $\Va$ is a variety of algebras, regarded as a category whose objects are the $\Va$-algebras and whose morphisms are the $\Va$-homomorphisms. 
\item $\U\colon \Va \to \Set$ is the 
 underlying-set functor.
\item $\F$ is the  free functor, i.e.\ the left adjoint to $\U$.
\item $A$ is an arbitrary but fixed $\Va$-algebra.
\end{itemize}
We  often speak  of `algebras' and `homomorphisms' (and also `isomorphisms' etc.) rather than `$\Va$-algebras' and `$\Va$-homomorphisms' etc., the variety $\Va$ being understood.

If $I$ is any set, the algebra $\F(I)$ in $\Va$ is, as usual,  \emph{a free algebra generated by $I$}. We fix canonical representatives for the isomorphism class of each free algebra in $\Va$. To this end, let 
\begin{align*}
X_\mu:=\{X_\alpha\}_{\alpha<\mu}
\end{align*}
be a specific set (of ``\emph{variables}'', or ``\emph{free generators}'') of cardinality $\mu$, where $\alpha$ ranges over ordinals (of cardinality) less than $\mu$. We often write $\mu$ as a shorthand for $X_{\mu}$, and therefore $\F(\mu)$
as a shorthand for $\F(X_\mu)$. To stress that we are selecting a specific representative for the  isomorphism class of a free algebra $\F(I)$, we refer to $\F(\mu)$ as \emph{the} free algebra on $\mu$ generators.

The adjunction relation
\begin{align}\label{eq:freeforget}
\frac{\F(\mu)\to A}{\mu\to \U(A)}
\end{align}
shows that $\U(A)$ may be naturally identified (in $\Set$) with the set of homomorphisms $\F(1)\to A$, i.e.\
\begin{align*}
\U{(A)}\cong \hom_\Va{(\F(1),A)}.
\end{align*}
In particular, because $\F$ is a left adjoint and therefore preserves all existing colimits, 
\begin{align}\label{eq:coprod}
\F(\mu)=\coprod_\mu\F(1)
\end{align}
i.e.\ $\F(\mu)$ is the coproduct in $\Va$ of $\mu$ copies of $\F(1)$.

With reference to the notation we adopted in the preceding sections, we now assume:
\begin{itemize}
\item $\T$ is the opposite of the full subcategory of $\Va$ whose objects are the free $\Va$-algebras $\F(\mu)$, as $\mu$ ranges over all cardinals.
\item $\S$ is the category $\Set$.
\item $\a$ is the $\T$-object $\F{(1)}$.
\end{itemize}
It remains to provide an instantiation for the functor $\I\colon \T\to \S$.  To this end notice that any algebra $A$ yields a functor
\begin{align*}
\I_{A}\cong \hom_{\Va}(-, A)\colon \T \to \Set
\end{align*}
that preserves arbitrary products, in the spirit of the Lawvere-Linton functorial semantics of algebraic theories \cite{lawvere63,Linton:1965,pareigis70, gabrielulmer71, manes76, adameketal94, adameketal11}; henceforth we  write simply $\I$ for $\I_{A}$.  We have
\begin{align}\label{eq:defI}
\I\left(\F(\mu)\right)\cong  \U(A)^\mu
\end{align}
for any $\mu$ and, given a homomorphism $\F(\mu)\to\F(\nu)$, the corresponding function $\U(A)^\nu\to\U(A)^\mu$ can be concretely described as follows. First, by (\ref{eq:coprod}), it suffices to consider the case  $\mu=1$. 
Thus, let 
\begin{align}\label{eq:element}
p\colon\F(1)\to\F(\nu)
\end{align}
be given. For an element of $\U(A)^{\nu}$, i.e.\ a function
\begin{align}\label{eq:pointofa}
a_\nu\colon \nu \to \U(A),
\end{align}
by the adjunction (\ref{eq:freeforget}) there is a unique $\Va$-arrow
\begin{align}\label{eq:tuple}
\widehat{a_\nu}\colon\F(\nu)\to A.
\end{align}
We then have the composition
\begin{align}\label{eq:comp}
\F(1)\overset{p}{\longrightarrow}\F(\nu)\overset{\widehat{a_\nu}}{\longrightarrow} A
\end{align}
of (\ref{eq:element}) and (\ref{eq:tuple}). Applying again the adjunction (\ref{eq:freeforget}) to (\ref{eq:comp}) we obtain an arrow in $\Set$
\begin{align*}
\ev(p,a_\nu):= 1 \to \U{(A)},
\end{align*}
i.e.\ an element of $\U({A})$, called the \emph{evaluation of $p$ at $a_\nu$}.
Keeping $p$ fixed and letting $a_{\nu}$ range over all elements (\ref{eq:pointofa}) of $\U{(A)}^{\nu}$, we thus obtain the \emph{evaluation map}
\begin{align}\label{eq:dualeval}
\ev(p,-)\colon \U(A)^\nu\to\U(A).
\end{align}
 We set 
\begin{align}\label{eq:Ionarrows}
\I(p):=\ev(p,-),
\end{align}
and this completes the definition of the functor $\I\colon \T\to\Set$. 
\begin{definition}\label{def:definablemap}A function $\U{(A)}^{\nu}\to \U{(A)}^{\mu}$ is called \emph{definable} \textup{(}\emph{in the language of $\Va$}\textup{)} if it is in the range of $\I$, as defined above. In other words, the definable functions $\U{(A)}^{\nu}\to \U{(A)}^{\mu}$ are precisely those that can be obtained by evaluating a $\mu$-tuple of elements of $\F{(\nu)}$ at the $\nu$-tuples of  elements of $A$.
\end{definition}

Observe that, in the above, $\I$ preserves all products in $\T$ by construction, so it satisfies Assumption \ref{ass:product-preserving}.  Note also that $\Set$ obviously satisfies Assumption \ref{ass:limits}. Concerning Assumption \ref{ass:cosep}, we have: 
\begin{lemma}\label{lem:algcosep}
The object $\a=\F{(1)}$ is an $\I$-coseparator for the functor $\I$ defined in \textup{(\ref{eq:defI}--\ref{eq:Ionarrows})} above. 
\end{lemma}
\begin{proof}We need to show that, for any cardinal $\mu$, the family of definable functions $f\colon A^{\mu} \to A$ is jointly monic in $\Set$. That is, given any two functions $h_1,h_2\colon S \to A^{\mu}$, if  $f\circ h_1=f\circ h_2$ for all definable $f$, then $h_1=h_2$. Note that the canonical projection functions $\pi_{\alpha}\colon A^{\mu}\to A$ of the product $A^{\mu}$, for $\alpha<\mu$ an ordinal, are definable. Indeed,  inspection of the definition of $\I$ shows that the unique homomorphism $\iota_{\alpha}\colon \F{(1)}\to\F{(\mu)}$ induced by $X_{1}\mapsto X_{\alpha}$ is such that $\I{(\iota_{\alpha})}=\pi_{\alpha}$. If now $h_{1}\neq h_{2}$, by the universal property of products there exists $\alpha<\mu$ with $\pi_{\alpha}\circ h_{1}\neq \pi_{\alpha}\circ h_{2}$, as was to be shown. 
\end{proof}

For the rest of this paper we follow the standard practice in algebra of omitting the underlying set functor when no confusion can arise. Thus we write, e.g., $a \in A$ in place of $a \in \U{(A)}$. Let us now consider the categories $\D$ and $\Q$ in the present algebraic setting.
Specialising the definitions in Subsection \ref{ss:D}, we see that the $\D$-objects  are all subsets $S\subseteq A^{\mu}$, as $\mu$ ranges over all cardinals. The $\D$-arrows from $S'\subseteq A^{\nu}$ to $S\subseteq A^{\mu}$ are the restrictions  $S' \to S$ of the definable functions  $A^{\nu}\to A^{\mu}$, in the sense of Definition \ref{def:definablemap}. 

Next, we  give an explicit algebraic presentation of the category $\Q$ defined in Subsection \ref{ss:R}. For $\nu$ a cardinal, a \emph{$\Va$-operation \textup{(or, more simply, an} operation\textup{)} of arity $\nu$} is a $\Va$-homomorphism $t\colon \F{(1)}\to\F{(\nu)}$. The operation $t$ is \emph{finitary} if $\nu$ is finite, and \emph{infinitary} otherwise. An \emph{operation on the $\Va$-algebra} $A$ is a function $h\colon A^{\nu}\to A$ that is definable in the sense of Definition \ref{def:definablemap}, that is, such that $h=\I{(t)}:=\ev{(t,-)}$ for some $t\colon \F{(1)}\to\F{(\nu)}$.
\begin{remark}\label{rem:operations}Since homomorphisms $t\colon \F{(1)}\to\F{(\nu)}$ are naturally identified with elements $t \in \F{(\nu)}$ via the adjunction (\ref{eq:freeforget}), the preceding definition  agrees with the usual  notion of  operations as \emph{term-definable functions}; one calls  $t$ a \emph{defining term} for the operation in question. By a classical theorem of  G.\ Birkhoff (see e.g.\ \cite[Theorem 10.10]{Burris:81})  the free algebra $\F{(\nu)}$ can indeed be represented as the algebra of \emph{terms}  ---elements of  absolutely free algebras---  over the set of variables $X_{\nu}$, modulo the equivalence relation that identifies two such terms if, and only if, they evaluate to the same element in any $\Va$-algebra. For the infinitary case see \cite[Ch.\ III]{Slominsky:1959}.
\end{remark}
 Direct inspection of the definitions shows that \emph{homomorphisms commute with  operations}:  given any $\Va$-homomorphism $h\colon A\to B$,  any $\nu$-ary operation $t \in \F{(\nu)}$, and any element $a_{\nu}:=(a_{\beta})_{\beta<\nu}\in A^{\nu}$, we have
	\begin{align}\label{eq:operation}
	h(\ev_{A}{(t,a_{\nu})})=\ev_{B}{(\,t, (h(a_{\beta}))_{\beta < \nu}}\,),
	\end{align}
	where $\ev_{A}(t,-)\colon A^{\nu} \to A$ and $\ev_{B}(t,-)\colon B^{\nu} \to B$ are the evaluation maps with respect to $A$ and $B$. It is common to write (\ref{eq:operation}) as
	\begin{align*}
	h(t(\,a_{\nu}\,))=t(\,(h(a_{\beta}))_{\beta<\nu}\,),
	\end{align*}
	where the algebras $A$ and $B$ over which $t$ is evaluated are tacitly  understood. A \emph{congruence}  $\theta$ on a $\Va$-algebra $A$ is an equivalence relation on $A$ that is \emph{compatible with \textup{(or} preserved by\textup{)} all operations}, i.e.\ with all definable maps $f\colon A^{\nu}\to A$ (i.e. maps of the form $f=\ev{(t,-)}$ for some defining term  $t \in \F{(\nu)}$), where $\nu$ is a cardinal. This means 
	that whenever $x_{\nu}:=(x_{\beta})_{\beta<\nu}$, $y_{\nu}:=(y_{\beta})_{\beta<\nu}$ are $\nu$-tuples of elements of $A$, 
	\begin{align*}
	(x_{\beta},y_{\beta})\in \theta \text{ for each $\beta<\nu$} \ \ \ \Longrightarrow \ \ \ (f(x_{\nu}), f(y_{\nu})) \in \theta.
	\end{align*}

\begin{remark}
	It is a standard fact, even in the infinitary case, that congruences as we just defined them coincide with congruences defined in terms of kernel pairs; see \cite[p.\ 33]{Linton:1965} and \cite[Ch.\ II.5]{Slominsky:1959}.
\end{remark}
\begin{lemma}\label{l:congruence}
	Consider any pair $(\F{(\mu)},R)$, where $R$ is an equivalence relation on
	 $\hom_{\T}(\F{(\mu)}, \a)\cong \U(\F{(\mu)})$. Then $R$ is a $\a$-congruence ---that is, $(\F{(\mu)}, R)$ is an object of $\Q$--- if, and only if, $R$ is a congruence on $\F{(\mu)}$.
\end{lemma}
\begin{proof}
	It suffices to notice that the definable arrows $\U{(\F{(\mu)})}^{\nu}\to \U{(\F{(\mu)})}$ correspond precisely to the arrows of the form $$t\circ -:\hom_{\T}(\F{(\mu)}, \a)^{\nu}\cong \hom_{\T}(\F{(\mu)}, \a^{\nu}) \to \hom_{\T}(\F{(\mu)}, \a)$$ under the natural identification $\U{(\F{(\mu)})}\cong \hom_{\T}(\F{(\mu)}, \a)$.
\end{proof}

We call a pair $(\F{(\mu)},\theta)$, for $\mu$ a cardinal and $\theta$ a congruence on $\F{(\mu)}$, a \emph{presentation} (\emph{in the variety $\Va$}). We call the algebra $\F{(\mu)}/\theta$ the \emph{algebra presented by} $(\F{(\mu)},\theta)$. We write $\Vap$ for the category of \emph{presented $\Va$-algebras}, having as objects all presentations in $\Va$, and as morphisms the $\Va$-homomorphisms between the $\Va$-algebras presented by them. Now $\Vap$ is $\Q$ to within an equivalence:
\begin{theorem}\label{thm:quotienteq}
	Let $\Va$ be any \textup{(}finitary or infinitary\textup{)} variety of algebras, $\Vap$ the associated category of presented $\Va$-algebras. Set $\T$ to be the opposite of the full subcategory of $\Va$ whose objects are the free $\Va$-algebras $\F{(\mu)}$, for $\mu$ an arbitrary cardinal, $\a:=\F{(1)}$, $\I\colon \T \to \Set$ to be the functor defined in Section \ref{sec:algebraic}, and $\Q$ to be the category defined as in Subsection \ref{ss:R}.
	Then, the categories $\Vap$ and $\Q$ are equivalent.
\end{theorem}
\begin{proof}
	Consider the functor that sends an object $(\F{(\mu)},\theta)$ in $\Vap$ into the object $(\F{(\mu)},\theta)$ of $\Q$ and any map $h\colon (\F{(\mu)},\theta)\to (\F{(\nu)},\theta')$ to itself. This functor is well-defined since by Lemma \ref{l:congruence} for any cardinal $\mu$ the congruences on $\F{(\mu)}$ correspond to the $\a$-congruences on $\hom_{\T}(\F{(\mu)}, \a)$  under the bijection $\U{(\F{(\mu)})}\cong \hom_{\T}(\F{(\mu)}, \a)$, and any homomorphism $\F{(\mu)}/\theta \to \F{(\nu)}/\theta'$ is induced by a homomorphism $\F{(\mu)}\to \F{(\nu)}$; moreover, it is clearly fully faithful and essentially surjective, whence it yields an equivalence of categories, as required. 
\end{proof}
\begin{remark}\label{rem:Vpequiv}The category $\Vap$ is equivalent to $\Va$.  Indeed, we have a functor that sends each presented algebra $(\F{(\mu)}, \theta)$ into the quotient $\F{(\mu)}/ \theta$ in $\Va$  and acts identically on maps.  It is elementary that this  functor is full, faithful, and essentially surjective.
\end{remark}
The Galois connections $(\CC,\VV)$ of Lemma \ref{lem:galois} now read as follows. Given a subset $S\subseteq A^{\mu}$, we have
\begin{align}\label{eq:Calg}
\CC{(S)}=\left\{(p,q)\in\F{(\mu)} \mid \forall a\in S: \ \ \ev{(p,a)}=\ev{(q,a)} \right\},
\end{align}
where $\ev(p,-)\colon A^{\mu}\to A$ is, once more, the evaluation map (\ref{eq:dualeval}).

Given a relation $R$ on $\F{(\mu)}$, we have
\begin{align}\label{eq:Valg}
\VV{(R)}=\bigcap_{(p,q)\in R}\left\{a\in\I{(\F{(\mu)})} \mid \ev{(p,a)}=\ev{(q,a)} \right\}.
\end{align}
Lemma \ref{lem:galois} asserts that, for any cardinal $\mu$, any relation $R$ on $\F{(\mu)}$, and any subset $S\subseteq A^{\mu}$, we have
\begin{align*}
R \subseteq \CC{(S)} \quad \quad \text{if, and only if,} \quad \quad S\subseteq \VV{(R)}.
\end{align*}
In other words, the functions $\VV\colon 2^{\F{(\mu)}^{2}}\to 2^{A^{\mu}}$ and $\CC\colon  2^{A^{\mu}}\to2^{\F{(\mu)}^{2}}$ yield a contravariant Galois connection between the indicated power sets.

Consider subsets $S'\subseteq A^{\nu}$, $S\subseteq A^{\mu}$, with $\mu$ and $\nu$ cardinals, and a $\D$-arrow $f\colon S'\subseteq A^{\nu}\to S\subseteq A^{\mu}$, i.e.\    a definable function $f\colon A^{\nu}\to A^{\mu}$ that restricts to a function $S'\to S$. 
Recall from  (\ref{eq:defI}--\ref{eq:Ionarrows}) that  $f$ is induced by a (uniquely determined) homomorphism $h\colon\F{(\mu)}\to \F{(\nu)}$ via evaluation. 
We have
\begin{align*}
\C{(S)}=(\F{(\mu)}, \CC{(S)})
\end{align*}
with $\CC{(S)}$ as in (\ref{eq:Calg}), and similarly for $S'$. The functor $\C$ carries the $\D$-arrow $f$ to the $\Q$-arrow $(\F{(\nu)}, \CC{(S')})\to(\F{(\mu)}, \CC{(S)})$ induced by the homomorphism $h\colon\F{(\mu)}\to \F{(\nu)}$.
Consider, conversely, $\Q$-objects $(\F{(\nu)},R')$ and $(\F{(\mu)}, R)$, together with a $\Q$-arrow $g:(\F{(\nu)},R')\to (\F{(\mu)}, R)$ induced  by a homomorphism, say $h\colon \F{(\mu)}\to\F{(\nu)}$. We have
\begin{align*}
\V{(\F{(\mu)}, R)}=\VV{(R)}\subseteq \I{(\F{(\mu)})}
\end{align*}
with $\VV{(R)}$ as in (\ref{eq:Valg}), and similarly for $(\F{(\nu)},R')$. The functor $\V$ sends $g$ to the restriction $S'\to S$ of the definable function $f\colon A^{\nu}\to A^{\mu}$ induced by $h$. As an immediate consequence of Theorem \ref{thm:weakadj} and Theorem \ref{thm:quotienteq}, we have:
\begin{corollary}[Algebraic affine adjunction]\label{cor:algadj} Consider any \textup{(}finitary or infinitary\textup{)} variety $\Va$ of algebras and its  associated category of presented algebras $\Vap$, and  fix any $\Va$-algebra $A$. Then the functors  $\C\colon  \D \to \Vap^{\rm op}$ and $\V\colon  \Vap^{\rm op}\to \D$ defined as in the above are adjoint with $\C\dashv \V$.\qed
\end{corollary}
\begin{remark}\label{extendedrepresentability}For any $\Va$-algebra $A$, the functor $\I\coloneqq\I_{A}\colon\T \to \Set$ extends to a representable functor on $\Va^{\rm op}$, namely 
\[
\I'\coloneqq\hom_{\Va}(-, A)\colon \Va^{\rm op} \longrightarrow \Set
\]
This $\I'$ can be seen as the extension of $\I$ along the functor ${I_{\a}}\colon \T \to \Q^{\rm op}$ that we looked at in Section \ref{s:concreteness}. (To replace $\Q$ with $\Va$ in the codomain of $I_{\a}$, use Theorem \ref{thm:quotienteq} along with Remark \ref{rem:Vpequiv}.) Then Proposition \ref{p:restriction} and Remark \ref{r:restriction} apply to show that the algebraic adjunction is the restriction of a larger  dual adjunction to which  Proposition \ref{p:repres} applies.
\end{remark}

\subsection{The algebraic \textsl{Nullstellensatz}}\label{s:algnull}
It is well known that in any (finitary or infinitary) variety  $\Va$ of algebras we have:
\begin{enumerate}
\item\label{it:into} The monomorphisms are exactly the injective $\Va$-homomorphisms, which we also call embeddings.
\item\label{it:onto} The regular epimorphisms (=the coequalisers of some pair of parallel arrows) are exactly the surjective  $\Va$-homomorphisms, which we also call quotient maps.
\end{enumerate} 
(See \cite[pp.\ 87--88]{Linton:1965}.) We shall use these basic facts often in this section.

Recall from Section \ref{sec:algebraic} that, for a cardinal $\nu$ and a given element $a\in A^{\nu}$, we have the homomorphism  $\widehat{a}\colon \F{(\nu)}\to A$ of \eqref{eq:tuple}.
Now,  the action of (the underlying function $\U{(\widehat{a})}$ of) $\widehat{a}$ is given by 
\begin{align}\label{eq:action}
p \in \F{(\nu)} \ \overset{\widehat{a}}{\longmapsto} \ \ev{(p,a)} \in A.
\end{align}
For, applying the adjunction  $\F\dashv\U$ to the arrow in \eqref{eq:comp}
\begin{align*}
\F(1)\overset{p}{\longrightarrow}\F(\nu)\overset{\widehat{a}}{\longrightarrow} A
\end{align*}
we obtain the commutative diagram
\begin{small}
\begin{align*}
\begin{tikzpicture}[scale=0.4]
\node (1) at (0,0)   {$1$};
\node (UF) at (7,0) {$\U{(\F{(\nu)})}$};
\node (U) at (14,0) {$\U{(A)}$};
\draw [->] (1) to  node [above, midway] {$\widecheck{p}$} (UF);
\draw [->] (UF) to   node [above, midway] {$\U{(\widehat{a})}$} (U);
\draw [->] (1) to [bend right]  node [below, midway] {$\ev{(p,a)}$} (U);
\end{tikzpicture}
\end{align*}
\end{small}
 where we write  $\widecheck{p}\colon 1 \to \U{(\F{(\nu)})}$ for the unique function corresponding to $p\colon \F{(1)}\to \F{(\nu)}$ under the adjunction.

We also have the natural quotient homomorphism
\begin{align}\label{eq:pointquotient}
q_{a}\colon \F{(\nu)}\twoheadrightarrow \F{(\nu)}/\CC{(\{a\})}.
\end{align}
By construction, $q_{a}$ preserves the relation $\CC{(\{a\})}$ on  $\F{(\nu)}$ with respect to the identity relation on $\F{(\nu)}/\CC{(\{a\})}$.
And $\widehat{a}$ preserves the relation $\CC{(\{a\})}$ on  $\F{(\nu)}$ with respect to the identity relation on $A$. Indeed, if $(p,q) \in \CC{(\{a\})}$ then, by definition, $\ev{(p,a)}=\ev{(q, a)}$, whence $\widehat{a}(p)=\widehat{a}(q)$ by (\ref{eq:action}).  Therefore, by the universal property of the quotient homomorphism there exists exactly one homomorphism 
\begin{align}\label{eq:gelfeval}
\gamma_{a}\colon \F{(\nu)}/\CC{(\{a\})} \longrightarrow A
\end{align} 
that makes the  diagram in Fig.\ \ref{fig:gelfand} commute.
\begin{figure}[h!]
\begin{small}
\smallskip
\begin{center}
\begin{tikzpicture}[scale=0.4]
\node  (F) at (0,5)   {$\F{(\nu)}$};
\node (A) at (0,0) {$A$};
\node (C) at (5,0) {$\frac{\F{(\nu)}}{\CC{(\{a\})}}$};
\draw [->] (F) -- (A) node [left, midway] {$\widehat{a}$};
\draw [->] (F) -- (C) node [right, midway] {$q_{a}$};
\draw [<-] (A) -- (C) node [below, midway] {$\gamma_{a}$} node [above, midway] {$!$};
\end{tikzpicture}
\end{center}
\end{small}
\caption{The Gelfand evaluation $\gamma_a$.}
\label{fig:gelfand}
\end{figure}
\begin{definition}[Gelfand evaluation]\label{def:gelfeval}Given a cardinal $\nu$ and an element $a\in A^{\nu}$,  we call the homomorphism in \textup{(\ref{eq:gelfeval})}  the \emph{Gelfand evaluation} (\emph{of $\F{(\nu)}$ at $a$}).
\end{definition}

\begin{lemma}\label{l:SGK}Fix a cardinal $\nu$.
\begin{enumerate}[\textup{(}i\textup{)}] 
\item\label{l:SGK:item1} For each  $a\in A^{\nu}$, the  Gelfand evaluation $\gamma_{a}$  is a monomorphism, and hence its is an injective function.
\item\label{l:SGK:item2} Conversely, for each congruence relation  $\theta$ on $\F{(\nu)}$, and each homomorphism $e\colon \F{(\nu)}/\theta\to A$,  consider the commutative diagram
\begin{small}
\begin{center}
\begin{tikzpicture}[scale=0.4]
\node (F) at (0,5)   {$\F{(\nu)}$};
\node (A) at (0,0) {$A$};
\node (C) at (5,0) {$\frac{\F{(\nu)}}{\theta}$};
\draw [->] (F) --  (A) node [left, midway] {$e\circ q_{\theta}$};
\draw [->] (F) --   (C) node [right, midway] {$q_\theta$};
\draw [<-] (A) --  (C) node [below, midway] {$e$};
\end{tikzpicture}
\end{center}
\end{small}
where $q_{\theta}$ is the natural quotient homomorphism. Set $a:=(e\circ q_\theta(X_\beta))_{\beta<\nu}\in A^\nu$. If $e$ is a monomorphism, then $\theta = \CC{(\{a\})}$, and the commutative diagram above coincides with the one in Fig.\ \ref{fig:gelfand}. \textup{(}That is, $q_\theta=q_a$, $e=\gamma_{a}$, and $e\circ q_\theta=\widehat{a}$.\textup{)}
\end{enumerate}
\end{lemma}
\begin{proof}
\noindent$(i)$\ It suffices to check that  $\gamma_{a}$ is injective. Pick $p,q \in \F{(\nu)}$ such that $(p,q)\not \in\CC{(\{a\})}$. Then, by definition, $\ev{(p,a)}\neq\ev{(q, a)}$, and therefore $\widehat{a}(p)\neq \widehat{a}(q)$ by (\ref{eq:action}). But then, by the definition of Gelfand evaluation,  it follows that $\gamma_{a}(p)\neq \gamma_{a}(q)$. 

\noindent$(ii)$\ Since $e$ is monic, we have $\ker{(e\circ q_\theta)}=\ker{q_\theta}=\theta$. Explicitly, 
\begin{align}\label{eq:kernel}
\forall s,t \in \F{(\nu)}: \ \ (s,t)\in\theta \ \  \Longleftrightarrow \ \ e(q_\theta(s))=e(q_\theta(t)).
\end{align}
By  the  definition of $a$, since homomorphisms commute with operations,  (\ref{eq:kernel}) yields
\begin{align}\label{eq:kerneltuple}
\forall s,t \in \F{(\nu)}: \ \ (s,t)\in\theta \ \ \Longleftrightarrow \ \ \ev{(s,a)}=\ev{(t,a)}.
\end{align}
Therefore, by (\ref{eq:kerneltuple}), we have $a \in \VV{(\theta)}$. By the Galois connection (\ref{eq:galois}) this is equivalent to $\theta\subseteq \CC{(\{a\})}$.

For the converse inclusion, if $(u,v)\in \CC{(\{a\})}$, then $\ev{(u,a)}=\ev{(v,a)}$, and therefore $(u,v)\in\theta$ by (\ref{eq:kerneltuple}). This proves $\theta=\CC{(\{a\})}$, and therefore $q_\theta=q_a$. To show $\widehat{a}=e\circ q_a$, note that, by the definition of $\widehat{a}$ and the universal property of $\F{(\nu})$, they both are the (unique) extension of the function $X_\beta \mapsto \ev{(X_\beta,(e\circ q_\theta(X_\beta)))}$, for $\beta < \nu$. 
\end{proof}
For a congruence relation $\theta$ on $\F{(\nu)}$, we now consider the natural quotient homomorphism
\begin{align}\label{eq:quotth}
q_{\theta}\colon \F{(\nu)} \to \F{(\nu)}/\theta,
\end{align}
together with the product $\prod_{a\in \VV{(\theta)}}\F{(\nu)}/\CC{(\{a\})}$ and its projections
\begin{align}\label{eq:prodsheaf}
\pi_{a}\colon \prod_{a\in \VV{(\theta)}}\frac{\F{(\nu)}}{\CC{(\{a\})}}\xrightarrow{\hskip .7cm }\frac{\F{(\nu)}}{\CC{(\{a\})}}.
\end{align}
We also consider the power $A^{\VV{(\theta)}}$ and its projections
\begin{align}\label{eq:prodfunc}
p_{a}\colon A^{\VV{(\theta)}}\longrightarrow A.
\end{align}
The morphisms (\ref{eq:gelfeval}--\ref{eq:prodfunc}) yield the commutative diagrams  ---one for each $a\in\VV{(\theta)}$---  in Fig.\ \ref{fig:birkgelf},
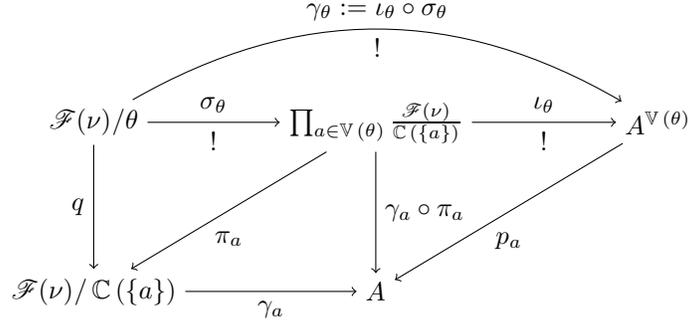
\begin{figure}[h!]
\centering
\smallskip
\begin{tikzpicture}[scale=0.75]
\node (Fth) at (0,0)   {$\F{(\nu)}/\theta$};
\node (P1) at (5,0) {$\prod_{a\in \VV{(\theta)}}\frac{\F{(\nu)}}{\CC{(\{a\})}}$};
\node (P2) at (10,0) {$A^{\VV{(\theta)}}$};
\node (FC) at (0,-3) {$\F{(\nu)}/\CC{(\{a\})}$};
\node (A) at (5,-3) {$A$};
\draw [->] (Fth) to node [left, midway] {$q$} (FC);
\draw [->] (P1) to node [right, midway] {$\gamma_{a}\circ\pi_{a}$} (A);
\draw [->] (FC) to node [below, midway] {$\gamma_{a}$} (A);
\draw [->] (P2) to node [below, midway,yshift=-0.15cm] {$p_{a}$} (A);
\draw [->] (P1) to node [below, midway,yshift=-0.15cm] {$\pi_{a}$} (FC);
\draw [->] (Fth) to node [above, midway] {$\sigma_{\theta}$} node [below, midway] {$!$} (P1) ;
\draw [->] (P1) to  node [above, midway] {$\iota_{\theta}$} node [below, midway] {$!$} (P2) ;
\draw [->] (Fth) to [bend left] node [above, midway] {$\gamma_{\theta}:=\iota_{\theta}\circ\sigma_{\theta}$} node [below, midway] {$!$} (P2);
\end{tikzpicture}
\caption{The Gelfand and Birkhoff transforms  $\gamma_\theta$ and $\sigma_\theta$.}
\label{fig:birkgelf}
\end{figure}
where $\sigma_{\theta}$ and $\iota_{\theta}$ are the unique homomorphisms whose existence is granted by the universal property of the products $\prod_{a\in \VV{(\theta)}}\frac{\F{(\nu)}}{\CC{(\{a\})}}$ and $A^{\VV{(\theta)}}$, respectively.
\begin{definition}[Gelfand and Birkhoff transforms]\label{eq:gelftrans}Given a cardinal $\nu$ and a congruence $\theta$ on $\F{(\nu)}$, the homomorphisms $\gamma_{\theta}:=\iota_{\theta}\circ\sigma_{\theta}$ and $\sigma_{\theta}$ given by the commutative diagram above are called the \emph{Gelfand} and the \emph{Birkhoff transforms} (\emph{of $\F{(\nu)}/\theta$ with respect to $A$}), respectively.\end{definition}
\begin{lemma}\label{lem:easygelf}With the  notation above, and for each $a \in A$, the homomorphisms $\pi_{a}\circ \sigma_{\theta}$ and $\iota_{\theta}$ are surjective and injective, respectively.
\end{lemma}
\begin{proof}It is clear that $\pi_{a}\circ \sigma_{\theta}$ is onto, because $q\colon \F{(\nu)}/\theta \to \F{(\nu)}/\CC{(\{a\})}$ is onto. Concerning 
$\iota_{\theta}$, let $x,y \in \prod_{a \in \VV{(a)}}\F{(\nu)}/\CC{(\{a\})}$, and suppose $\iota_{\theta}(x)=\iota_{\theta}(y)$.  With reference to the commutative diagram in Fig.\ {\ref{fig:birkgelf}},  for each $a \in \VV{(\theta)}$ we have $p_a(\iota_\theta(x))=p_a(\iota_\theta(y))$, and therefore $\gamma_a(\pi_a(x))= \gamma_a(\pi_a(y))$. Since $\gamma_{a}
$ is a monomorphism for each $a$ by Lemma \ref{lem:easygelf}, we infer $\pi_a(x)=\pi_a(y)$ for each $a$, and hence $x=y$ by the universal property of the product $\prod_{a 
\in \VV{(a)}}\F{(\nu)}/\CC{(\{a\})}$.
\end{proof}
\begin{definition}[Radical]\label{def:radical}For a cardinal $\nu$ and a relation $R$ on $\F{(\nu)}$, we call the congruence
\[
\bigcap_{a \in \VV{(R)}} \CC{(\{a\})}
\]
the \emph{radical  of $R$ \textup{(}with respect to the $\Va$-algebra $A$\textup{)}}. A congruence $\theta$ on $\F{(\nu)}$ is \emph{radical \textup{(}with respect to $A$\textup{)}} if $\theta=\bigcap_{a \in \VV{(\theta)}} \CC{(\{a\})}$.
\end{definition}
Note that the inclusion
\begin{align*}
\theta \subseteq \bigcap_{a \in \VV{(\theta)}} \CC{(\{a\})},
\end{align*}
always holds, cf.\ \eqref{eq:contained2}.

\begin{theorem}[Algebraic \textsl{Nullstellensatz}]\label{thm:algnull}
For any $\Va$-algebra $A$, any cardinal $\nu$, and any congruence $\theta$ on $\F{(\nu)}$. The following are equivalent.
\begin{enumerate}[\textup{(}i\textup{)}]
\item\label{l:null-subdirect-item1} $\CC{(\VV{(\theta)})}=\theta$.
\item\label{l:null-subdirect-item2} $\theta=\bigcap_{a \in \VV{(\theta)}} \CC{(\{a\})}$, i.e.\ $\theta$ is a radical congruence with respect to $A$.
\item\label{l:null-subdirect-item3} The Birkhoff transform $\sigma_{\theta}\colon\frac{\F{(\nu)}}{\theta}\longrightarrow\prod_{a\in \VV{(\theta)}}\frac{\F{(\nu)}}{\CC{(\{a\})}}$ is a subdirect embedding. 
\end{enumerate}
\end{theorem}

\begin{remark}In the proof that follows we apply three standard results in universal algebra, namely, \cite[Theorems 7.15, 6.15, and 6.20]{Burris:81}. Although in \cite{Burris:81} these  results are stated and proved for finitary varieties, the same proofs work for infinitary ones.
\end{remark}
\begin{proof}
\noindent The hypotheses of Theorem \ref{thm:null} are satisfied: the terminal object in $\Set$ is a singleton $\{a\}$, and the family of functions $\{a\}\to \VV{(R)}$ ---i.e.\ the elements of $\VV{(R)}$--- is obviously jointly epic. This proves the equivalence of  (\ref{l:null-subdirect-item1}) and (\ref{l:null-subdirect-item2}).

\noindent $(\ref{l:null-subdirect-item2})\Leftrightarrow (\ref{l:null-subdirect-item3})$.  By  \cite[Theorem 7.15]{Burris:81}, given any algebra $B$ and a family $\{\theta_{i}\}_{i\in I}$ of congruences  on $B$, the natural homomorphism $h\colon B\to\prod_{i\in I}B/\theta_{i}$ induced by the quotient homomorphisms $q_{\theta_{i}}\colon B\to B/{\theta_{i}}$ is an embedding if, and only if, $\bigcap_{i\in I}\theta_{i}$ is the identity congruence $\Delta$ on $B$. Taking 
\[B:=\F{(\nu)}/\theta\  \text{ and } \  \{\theta_{i}\}:=\{\CC{(\{a\})}\}_{a \in \VV{(\theta)}},\] we  obtain  the natural homomorphism 
\begin{align}\label{eq:nath}
h\colon \F{(\nu)}/\theta\longrightarrow  \prod_{a\in \VV{(\theta)}}\frac{\F{(\nu)}/\theta}{\CC{(\{a\})}/\theta},
\end{align}
where $\CC{(\{a\})}/\theta$ denotes the set $\{(p/\theta,q/\theta)\in \F{(\nu)}/\theta\mid (p,q) \in \CC{(\{a\})}\}$, which is easily seen to be a congruence relation on $\F{(\nu)/\theta}$.
It is clear by construction  that if $h$ is an embedding, then it is subdirect. Hence we have:
\begin{align}\label{eq:subdirect}
h\text{ is a subdirect embedding} \ \Longleftrightarrow\ \bigcap_{a\in \VV{(\theta)}}\CC{(\{a\})}/\theta=\Delta/\theta\
\end{align}
For each $a\in \VV{(\theta)}$, by the Galois connection \eqref{eq:galois}   we have $\theta\seq \CC{(\{a\})}$. Therefore, by the second isomorphism theorem \cite[Theorem 6.15]{Burris:81},
\begin{align}\label{eq:secondiso}
\forall a \in \VV{(a)}: \ \ \frac{\F{(\nu)}/\theta}{\CC{(\{a\})}/\theta}\cong\F{(\nu)}/\CC{(\{a\})}. 
\end{align}
From (\ref{eq:subdirect}--\ref{eq:secondiso}) we see:
\begin{align*}
h \text{ is a subdirect embedding} \ \ \Longleftrightarrow \ \ \sigma_{\theta} \text{ is a subdirect embedding.}
\end{align*}
Finally, upon recalling that, by \cite[Theorem 6.20]{Burris:81}, the mapping $\theta'\mapsto \theta'/\theta$ is an isomorphism of lattices between the lattice of congruences of $\F{(\nu)}$ extending $\theta$ and the lattice of congruences of $\F{(\nu)}/\theta$, we have 
\begin{align}\label{eq:lattcong}
\bigcap_{a\in \VV{(\theta)}}\CC{(\{a\})}/\theta=\Delta/\theta \ \ \Longleftrightarrow \ \ \bigcap_{a\in\VV{(\theta)}}\CC{(\{a\})}=\theta. 
\end{align}
In conclusion, (\ref{eq:nath}--\ref{eq:lattcong}) amount to the equivalence between $(\ref{l:null-subdirect-item2})$ and $(\ref{l:null-subdirect-item3})$.
\end{proof}

\begin{remark}\label{r:finitary-vs-inifinitary}
Since Birkhoff's influential paper  \cite{birkhoff1944subdirect} the theory of  algebras definable by operations of finite arity only has been  developed intensively. In  \cite[Theorem 1]{birkhoff1944subdirect} Birkhoff pointed out, by way of motivation for his main result, that the Lasker-Noether theorem \cite{noether21} generalises easily to algebras whose congruences satisfy the ascending chain condition, even in the presence of operations of infinite arity. His main result \cite[Theorem 2]{birkhoff1944subdirect} then showed how to extend the Lasker-Noether theorem to any variety of algebras, without any chain condition, provided however that all operations be finitary. In short, Birkhoff's subdirect representation theorem (see e.g.\ \cite[Theorem 2.6]{jacobson80} for a textbook treatment) fails for infinitary varieties of algebras. Much of the remaining general theory, however, carries over to the infinitary case. The two classical references on infinitary varieties are \cite{Slominsky:1959, Linton:1965}. Linton's paper \cite{Linton:1965}, in particular, extended Lawvere's categorical treatment of universal algebra \cite{lawvere63, lawverereprint}.
\end{remark}
Closing a circle of ideas, let us indicate how to obtain a ring-theoretic {\it Nullstellensatz} from Theorem \ref{thm:algnull}. Continuing the notation in the Introduction, we consider an algebraically closed field $k$, and finitely many variables $X:=\{X_{1},\ldots,X_{n}\}$, $n\geq 0$ an integer. The  $k$-algebra freely generated by $X$ is 
the polynomial ring $k[X]$. Congruences on any $k$-algebra are in one-one inclusion-preserving correspondence with ideals. We  let $\Va$ be the variety of $k$-algebras,  and we let  $A:=k$. The details then depend on what definition one takes for the notion of radical ideal. We shall use:
\begin{definition}\label{def:radideal}An ideal of a $k$-algebra is \emph{radical} if, and only if, it is an intersection of maximal ideals.
\end{definition}
We use  a classical result from commutative algebra; see e.g.\ \cite{atyiahmacdonald}.

\begin{lemma}[Zariski's Lemma]\label{lem:zariski} Let $F$ be any field, and suppose $E$ is a finitely generated $F$-algebra that is itself a field. Then $E$ is a finite field extension of $F$.\qed
\end{lemma}
The conjunction of   Lemma \ref{l:SGK} with Zariski's Lemma entails the following  fact, which is the key result in the ring-theoretic setting:
\begin{lemma}[Lemma of Stone-Gelfand-Kolmogorov type]\label{l:SGK-for-rings}An ideal $I$ of $k[X]$ is maximal if, and only if, there exists $a \in k^{n}$ such that $I=\CC{(\{a\})}$.
\end{lemma}
\begin{proof}
Assume $I=\CC{(\{a\})}$, and consider the Gelfand evaluation $\gamma_{a}\colon k[X]/\CC{(\{a\})}$ $\to k$ of Definition \ref{def:gelfeval}. By Lemma \ref{l:SGK}, $\gamma_{a}$ is an embedding. From the fact that $\gamma_{a}$ is a homomorphism of $k$-algebras it follows at once that it is onto $k$, and hence an isomorphism. Moreover $k$, being a field, is evidently simple in the universal-algebraic sense, i.e.\ it has no non-trivial ideals. Hence $\CC{(\{a\})}$, the kernel of the homomorphism $q_{a}\colon k[X]\to k[X]/I$ as in (\ref{eq:pointquotient}), is maximal (by direct inspection, or using the more general \cite[Theorem 6.20]{Burris:81}).

Conversely, assume that $I$ is maximal, and consider the natural quotient map $q_{I}\colon k[X]\to k[X]/I$. Then $k[X]/I$ is a simple finitely generated $k$-algebra, and hence a field. By Zariski's Lemma \ref{lem:zariski}, $k[X]/I$ is a finite field extension of $k$; since $k$ is algebraically closed, $k$ and $k[X]/I$ are isomorphic. Applying  Lemma \ref{l:SGK} with $e\colon k[X]/I\to k$ the preceding isomorphism completes the proof.
\end{proof}
\begin{corollary}[Ring-theoretic {\it Nullstellensatz}]\label{c:ringnull}For any ideal $I$ of $k[X]$, the following are equivalent.
\begin{itemize}
\item[\textup{(}i\textup{)}] $\CC{(\VV{(I)})}=I$.
\item[\textup{(}ii\textup{)}] $I$ is radical.
\end{itemize}
\end{corollary}
\begin{proof}Immediate consequence of Lemma \ref{l:SGK-for-rings} together with Theorem \ref{thm:algnull}.
\end{proof}
It is now possible to functorialise the above along the lines of the affine adjunctions studied in this paper, thereby obtaining the usual classical algebraic adjunction. We do not spell out  the details.

\section{The setting of syntactic categories}\label{syntacticcategories}

It is important to remark that the abstract categorical framework  developed in Section \ref{s:affadj} can be applied well beyond the standard setting of universal algebra investigated in Section \ref{sec:algebraic}. For the rest of this section, we assume that the reader is familiar with the notions of syntactic category and of classifying topos of a geometric theory. For background, please see  \cite{MR} and the first two chapters of \cite{CaramelloBook}. 

Recall that the syntactic category ${\mathcal{C}}_{\mathbb T}$ of a geometric theory $\mathbb T$ has as objects the geometric formulas-in-context $\{\vec{x}. \phi\}$ over the signature of the theory, and as arrows $\{\vec{x}. \phi\} \to \{\vec{y}. \phi\}$ the $\mathbb T$-provable equivalence classes $[\theta]$ of geometric formulas $\theta(\vec{x}, \vec{y})$ which are $\mathbb T$-provably functional from $\phi(\vec{x})$ to $\psi(\vec{y})$. The notion of $\mathbb T$-provably functional formula naturally generalises the notion of (morphism defined by) a term; indeed, for any term $t(\vec{x})$, the formula $y=t(\vec{x})$ is provably functional from $\{\vec{x}. \top\}$ to $\{\vec{y}. \top\}$. For any geometric theory $\mathbb{T}$, the models of $\mathbb T$ in any Grothendick topos $\mathcal{E}$ can be identified with functors ${\mathcal{C}}_{\mathbb T}\to \mathcal{E}$ preserving the geometric structure on the category ${\mathcal{C}}_{\mathbb T}$. The functor $F_{M}:{\mathcal{C}}_{\mathbb T}\to \mathcal{E}$ corresponding to a $\mathbb T$-model $M$ in $\mathcal{E}$ sends $\{\vec{x}. \phi\}$ to (the domain of) its interpretation $[[\vec{x}. \phi]]_{M}$ in $M$ and any arrow $[\theta]:\{\vec{x}. \phi\} \to \{\vec{y}. \psi\}$ in $\mathcal{C}_{\mathbb T}$ to the arrow $[[\vec{x}. \phi]]_{M}\to [[\vec{y}. \psi]]_{M}$ in $\mathcal{E}$, denoted by $[[\theta]]_{M}$ abusing notation, whose graph is the interpretation of the formula $\theta$ in $M$. 

A natural context for applying the categorical framework of Section \ref{s:affadj}  is obtained, for a given geometric theory $\mathbb T$, by taking:

\begin{itemize}
	\item $\T$ to be a full subcategory of $\mathcal{C}_{\mathbb T}$ closed (in $\mathcal{C}_{\mathbb T}$) under all  powers of $\{x. \top\}$;
	
	\item $\a$ to be the object $\{x. \top\}$;
	
	\item $\S$ to be  a Grothendieck topos (for instance, the category $\Set$ of sets);
	
	\item $\mathscr{I}$ to be the functor $F_{M}\colon\T\to \S$ corresponding to an arbitrarily fixed $\mathbb T$-model $M$ in $\S$ as specified above. 
\end{itemize}
We now check that the above instantiation meets the requirements of Assumptions \ref{ass:limits},  \ref{ass:product-preserving}, and \ref{ass:cosep}.

Clearly, if $\T$ is closed under all  powers of $\{x. \top\}$ existing in $\mathcal{C}_{\mathbb T}$ then $\I$ preserves all the existing powers of $\a$, so that Assumption \ref{ass:product-preserving} is satisfied.

Assumption \ref{ass:limits} is satisfied as any Grothendieck topos $\S$ has small limits. We next verify the remaining requirement that $\a$ is an $\I$-coseparator.

One can consider congruences and definable maps in the general setting of (one-sorted) geometric theories:

\begin{definition}\label{defcongruence}
	Let $\mathbb T$ be a one-sorted geometric theory and $M$ a $\mathbb T$-model.
	\begin{enumerate}[(a)]
		\item A \emph{definable map} $M^{k}\to M$ is a map of the form $[[\theta]]_{M}$ where $\theta$ is a $\mathbb T$-provably functional formula from $\{x. \top\}^{k}$ to $\{x. \top\}$.
		\item  A \emph{congruence} on $M$ is an equivalence relation $R$ on $M$ such that for any definable map $d:M^{k}\to M$, $(x_{i}, y_{i})\in R$ for all $i=1, \ldots, k$ implies that $(d(x_{1}, \ldots, x_{k}), d(y_{1}, \ldots, y_{k}))\in R$.
	\end{enumerate}
\end{definition}

\begin{lemma}
	In the setting defined above, the object $\a$ is always an $\I$-coseparator.
\end{lemma} 

\begin{proof}
	We have to verify that for every object $\{\vec{x}. \phi\}$ of $\T$, the family of arrows $[[\theta]]_{M}:[[\vec{x}. \phi]]_{M}\to M$, where $\theta$ varies among the $\mathbb T$-provably functional formulas from $\{\vec{x}. \phi\}$ to  $\{y. \top\}$, is jointly monic in $\S$. Now, if $\vec{x}=(x_{1}, \ldots, x_{n})$, for any $i\in \{1, \ldots, n\}$ the formula $y=x_{i}\wedge \phi(\vec{x})$ is $\mathbb T$-provably functional from $\{\vec{x}. \phi\}$ to  $\{y. \top\}$. But the interpretations in $M$ of such formulas are nothing but the canonical projections $[[\vec{x}. \phi]]_{M}\subseteq M^{n}\to M$, which are obviously jointly monic.
\end{proof}

\begin{remarks}\label{rem_syntactic}
	\begin{enumerate}[(a)]
		\item If $\mathbb T$ is an algebraic theory, it is natural to take $\T$ equal to the subcategory of ${\mathcal{C}_{\mathbb T}}$ whose objects are the powers of $\{x. \top\}$ (that is, the opposite of the category of free $\mathbb T$-models. One can prove that the $\mathbb T$-functional formulas between formulas are all induced by terms, up to $\mathbb T$-provable equivalence (see e.g.\ \cite[p.\ 120]{blasce}). In particular, the notions   in Definition \ref{defcongruence} agree with the corresponding algebraic ones.
		
		\item If $M$ is a \emph{conservative model} for $\mathbb T$ (i.e. any geometric sequent over the signature of $\mathbb T$ which is valid in $M$ is provable in $\mathbb T$) then the arrows $\{x. \top\}^{k} \to \{x. \top\}$ in $\T$ can be identified with the $\mathbb T$-model definable maps $M^{k}\to M$, for each $k$. Indeed, for any two $\mathbb T$-provably functional formulas $\theta_{1}, \theta_{2}$ from $\{x. \top\}^{k}$ to $\{x. \top\}$, we have $[\theta_{1}]=[\theta_{2}]$ if and only if $\theta_{1}$ and $\theta_{2}$ are $\mathbb T$-provably equivalent; but this is equivalent, $M$ being conservative, to the condition that $[[\theta_{1}]]_{M}=[[\theta_{2}]]_{M}$. For example, if $\mathbb T$ is the algebraic theory of Boolean algebras, the Boolean algebra $\{0,1\}$ is a conservative model for $\mathbb T$ and hence the free Boolean algebra on $k$ generators can be identified with the set of definable maps $\{0,1\}^{k}\to \{0,1\}$.
	\end{enumerate} 
\end{remarks}

A particularly natural class of theories to which the above setting can be applied is that of theories of presheaf type. Recall that a geometric theory is said to be of \emph{presheaf type} if it is classified by a presheaf topos. This class of theories includes all finitary algebraic (or, more generally, cartesian) theories and many other interesting cases, such as the theory of total orders, the theory of algebraic extensions of a base field, the theory of lattice-ordered abelian groups with strong unit \cite{CaramelloRusso1}, the theory of perfect MV-algebras or more generally of local MV-algebras in a proper variety of MV-algebras (see \cite{CaramelloRusso2} and \cite{CaramelloRusso3}), etc.
In fact, theories of presheaf type represent the `logical counterpart' of small categories: \emph{every} small category is, up to idempotent-splitting completion, the category of finitely presentable models of a theory of presheaf type (cf.\ \cite[Section 6.1]{CaramelloBook}). For a comprehensive study of this class of theories, we refer the reader to  \cite[Chapter 6]{CaramelloBook}.

Interestingly, while free objects in the category of set-based models of a theory of presheaf type $\mathbb T$ need not exist,  the category is always generated by the finitely presentable (equivalently, finitely presented) $\mathbb T$-models. The full subcategory spanned by such models is dual to the full subcategory of the syntactic category of the theory $\mathbb T$ on the $\mathbb T$-irreducible formulas (cf.\  \cite[Theorem 4.3]{CaramelloSyntactic}), and for each such formula $\phi(\vec{x})$ presenting a model $M_{\phi}$, we have (if $\mathbb T$ is one-sorted) $M_{\phi}\cong \hom_{{\mathcal C}_{\mathbb T}}(\{\vec{x}. \phi\}, \{x. \top\})$ (cf.  \cite[Theorem 6.1.17]{CaramelloBook}).

\begin{proposition}\label{stability}
	Let $\mathbb T$ be a one-sorted theory of presheaf type, and suppose $\{\vec{x}.\phi\}$ is an object of $\T$ presenting a $\mathbb T$-model $M$. Then congruences on $M$ in the sense of Definition \ref{defcongruence}\textup{(}b\textup{)} are precisely $\a$-congruences on $M$ \textup{(}regarding $M$ as $\hom_{\T}(\{\vec{x}. \phi\}, \a\textup{)}$ in the sense of section \ref{ss:FromRtoCong}.
\end{proposition}
\begin{proof}
	It suffices to notice that, for any arrow $[\theta]:\a^{k} \to \a$ in $\S$, $[[\theta]]_{M}$ is precisely the function $[\theta]\circ -:\hom_{\T}(\{\vec{x}. \phi\}, \a)^{k}\cong \hom_{\T}(\{\vec{x}. \phi\}, \a^{k}) \to \hom_{\T}(\{\vec{x}. \phi\}, \a)$.
\end{proof}
Hence, taking $\T$ to be the full subcategory of the geometric syntactic category $\mathcal{C}_{\mathbb T}$ of a theory of presheaf type $\mathbb T$ on the formulas which are either $\{y. \top\}$ or $\mathbb T$-irreducible and $\S$ to be $\Set$ yields in particular an adjunction between a category of congruences on finitely presentable $\mathbb T$-models and a certain category of definable sets and $\mathbb T$-definable maps between them. We also note that the equivalence between the first two items in the algebraic \emph{Nullstellensatz} (Theorem \ref{thm:algnull}) holds more generally for any theory of presheaf type $\mathbb T$ (replacing $\F(\mu)$ with any finitely presentable $\mathbb T$-model), with essentially the same proof. This shows in particular that one can construct natural analogues of the affine algebraic adjunctions of Section \ref{sec:algebraic} also in non-algebraic settings.

\section{Two classical examples}\label{p:dualities}
In this section we indicate how the algebraic affine adjunction of Section \ref{sec:algebraic}  relates to two classical duality theories, Stone duality for Boolean algebras and Gelfand duality for commutative unital $C^*$-algebras. In both cases we will see that the characterisation of the fixed points on the algebraic side of the adjunction is streamlined by the {\it Nullstellensatz}. By contrast,  at this stage we do not have an informative characterisation of the fixed points of the adjunction on the geometric side. Nonetheless, some preliminary remarks are possible, and we begin with these.
\subsection{Remarks on the dual \textsl{Nullstellensatz}}\label{topNullstellensatz}Considering again  the setting of  an arbitrary variety, we continue the notation
of Section \ref{sec:algebraic}. Let us henceforth  assume that the operator $\VC$ is topological, i.e.\ it commutes with finite unions. This assumption does hold for a number of classical dualities, including in particular  Stone and Gelfand dualities.  

The set $A$ is naturally equipped with the topology whose closed sets are those of the form $\VC{(S)}$, for $S\seq A$. We call this the \emph{Zariski} topology.  There are then at least two choices for a topology on the power $A^{\mu}$.
\begin{enumerate}
\item The product topology generated by the Zariski topology on $A$.
\item The \emph{Zariski} topology on $A^{\mu}$, i.e., the one in which the closed sets are those of the form $\VC{(T)}$, for $T\seq A^{\mu}$.
\end{enumerate}	
The product topology is easier to work with, while the Zariski topology yields the (tautological) characterisation of the subsets fixed by the Galois connection as precisely the closed subsets.  It is reasonable to ask for a  criterion that ensures the agreement of the two topologies.  We begin with an observation.
\begin{lemma}\label{l:continuous}
Definable functions $A^\mu\to A$ are continuous with respect to the Zariski topology.
\end{lemma}
\begin{proof}
Let $f$ be a definable function from $A^{\mu}$ into $A$, with defining term $\lambda(X_{\mu})$, with $X_{\mu}=(X_{\alpha})_{\alpha<\mu}$.  We prove that the inverse image of every closed set in $A$ is closed in $A^{\mu}$.  Arbitrary closed sets in $A$ are of the form $\VC{(S)}$ for  $S\seq A$.  Consider the set 
\[\theta\coloneqq\{\left( s(\lambda(X_{\mu})), t(\lambda(X_{\mu}))\right )\mid (s,t)\in \CC(S)\};\] 
we claim that $f^{-1}[\VC{(S)}]=\VV(\theta)$.  Indeed $d_{\mu}\in f^{-1}[C]$ if, and only if, there is $c\in \VC(S)$ such that $f(d_{\mu})=c$. By definition $c\in \VC(S)$ if, and only if, for all $(s,t)\in \CC(S)$, $s(c)=t(c)$.  Now, since $\lambda$ is a defining term for $f$, we have  $c=f(d_{\mu})=\lambda(d_{\mu})$.  So for all  $(s,t)\in \CC(S)$, $s(\lambda(d_{\mu}))=t(\lambda(d_{\mu}))$, and this is equivalent to saying that $d_{\mu}\in \VV(\theta)$.
\end{proof}
This leads to the following.
\begin{lemma}\label{l:co-null}
Assume that the Zariski topology on $A$ is Hausdorff and that each definable function $A^\mu\to A$ is continuous with respect to the product topology on the domain and the Zariski topology on the codomain. Then the product and the Zariski topologies  on $A^\mu$ coincide.
\end{lemma}
\begin{proof} 
By definition  the product topology is the coarsest topology that makes  all projections continuous \cite[pag.\ 89]{kelley1955general}, and  projections are clearly definable; hence  Lemma \ref{l:continuous} guarantees that the product topology is always coarser than the Zariski topology.  For the converse inclusion,  if $X$ is any space, and $Y$ is Hausdorff, then for any two continuous functions $f,g\colon X\to Y$ the solution set of the equation $f=g$ is closed in $X$ \cite[1.5.4]{engelking}. By assumption $A$ is Hausdorff in the Zariski topology, and  definable functions are continuous with respect to the product topology on $A^\mu$ and the Zariski topology on $A$, so for any pair of terms $(s,t)$ the set $\VV{(s,t)}$ is closed in the product topology.  On the other hand, $\VV(R)=\VV{\big(\bigcup_{(s,t)\in R}\{(s,t)\}\big)}=\bigcap_{(s,t)\in R}\VV{(s,t)}$ holds by Lemma \ref{lem:galois}. We conclude that $\VV(R)$ is closed in the product topology of $A^\mu$ for any subset $R$ of $\F_{\mu}\times\F_{\mu}$. 
\end{proof}
\begin{remark}\label{c:discrete-topology}
Under the further assumption that the variety $\Va$ is finitary, Lemma \ref{l:co-null} affords a connection with the theory of natural dualities \cite{clark1998natural}.  In that setting it is common to work with a finite dualising object equipped with the discrete topology, and  to use the product  topology on powers to  characterise the dual spaces. In our algebraic framework, too, the discreteness assumption on $A$   simplifies matters.  Indeed, if the Zariski topology on $A$ is discrete, so that all  finite products are also discrete, then since the variety is finitary the definable functions $A^\mu\to A$ are continuous with respect to the product topology on $A^{\mu}$, for any cardinal $\mu$.  Thus the assumptions of Lemma \ref{l:co-null} are met, ensuring that the Zariski topology and the product topology coincide.
\end{remark}

\subsection{Stone duality for Boolean algebras}\label{s:stone}
We indicate how to derive Stone duality for the variety  $\Va$ of Boolean algebras from the general adjunction. 
Take $A$ to be the two-element  Boolean algebra $\{0,1\}$. The only  piece of information  we need about $\Va$ to characterise the algebras fixed by the adjunction is the following.
\begin{lemma}[\mbox{\cite[Lemma 1]{birkhoff1944subdirect}}]\label{l:sub-boole}
To within an isomorphism, the only subdirectly irreducible Boolean algebra is  $\{0,1\}$.
\end{lemma}
By Corollary \ref{cor:algadj} we have a dual adjunction between $\Vap$ and $\D$ given by the functors $\C$ and $\V$.  We are interested in characterising the fixed points of this adjunction.  
\begin{lemma}\label{c:fix-alg-boole}
With $\Va$ and $A$ as above and  with reference to the functors of Corollary \ref{cor:algadj}, one has that all algebras in $\Vap$ are fixed by the composition $\C\circ\V$.
\end{lemma}
\begin{proof}
By Lemma \ref{l:sub-boole} all subdirectly irreducible Boolean algebras embed into $A$, hence by Lemma \ref{l:SGK} (item \ref{l:SGK:item2}), all congruences presenting a subdirectly irreducible Boolean algebra are of the form $\CC{(a)}$ for a suitable $a$.  Since $\Va$ is a finitary variety, the  congruences that present a subdirectly irreducible algebra are exactly the completely meet irreducible ones in the lattice of congruences of $\F{(\mu)}$ (see e.g.\  \cite[Lemma 4.43]{MMT87}).  Since  lattices of congruences are algebraic, all congruences are intersections of completely meet irreducible elements (see e.g.\  \cite[Theorem 2.19]{MMT87}), hence by the algebraic Nullstellensatz (Theorem \ref{thm:algnull}) we obtain that all algebras in $\Vap$ are fixed by the composition $\C\circ\V$.
\end{proof}

We now turn to the side of affine subsets.  The category $\D$ is given by  subsets of $\{0,1\}^{\mu}$, for $\mu$ ranging among all cardinals, and definable maps between them.  
\begin{lemma}\label{l:fix-top-boole}The closure operator $\CC\circ\VV$ is topological. The Zariski topology on $\{0,1\}$ is the discrete topology.
For any cardinal $\mu$,  a set $S\seq\{0,1\}^{\mu}$ is closed in the product topology if, and only if,  $\VV{(\CC{(S)})}= S$.
\end{lemma}
\begin{proof} The first statement is an easy direct verification, using for instance the $\wedge$ operation of Boolean algebras.
The Zariski topology on $A=\{0,1\}$ is discrete as $\{0\}=\VV{(0,x)}$ and $\{1\}=\VV{(1,x)}$. The second statement follows from  Lemma \ref{l:co-null} and Remark \ref{c:discrete-topology}.
\end{proof}
Combining  Lemma \ref{c:fix-alg-boole} with Lemma \ref{l:fix-top-boole},   we obtain: 
\begin{corollary}\label{c:presentedBoole-closedsubsets}
Let $\Va$ be the variety of Boolean algebras and their homomorphisms, and let $A$ be the Boolean algebra $\{0,1\}$. The functors of Corollary \ref{cor:algadj} form a contravariant equivalence between $\Va$ and the category of closed subspaces of the generalised Cantor cubes $\{0,1\}^{\mu}$ with definable maps between them.
\end{corollary}

To recover Stone duality in its classical form, further specific work is needed on the spatial side. Specifically, one needs (a)\ an intrinsic characterisation of the closed subspaces of $\{0,1\}^{\kappa}$,  and (b)\ an intrinsic characterisation of the definable maps between closed subspaces of generalised Cantor cubes.  The answers in both cases are well known:  the closed subspaces of  $\{0,1\}^{\kappa}$, for $\kappa$ ranging among cardinals,  are exactly the compact, Hausdorff, zero-dimensional spaces; and the definable maps between such subspaces are precisely the continuous maps.

\subsection{Gelfand duality for  $C^{*}$-algebras}\label{s:stone-gelfand}
A \emph{\textup{(}complex, commutative, unital\textup{)} $C^*$-algebra} is a complex commutative Banach algebra $A$ (always unital, with identity element written $1$) equipped with an involution ${\cdot}^*\colon A\rightarrow A$ satisfying $\|x^*x\|=\|x\|^2$ for each $x\in A$. Henceforth, `$C^*$-algebra' means `complex commutative unital $C^*$-algebra'. The category $\Cst$ has as objects 
 $C^*$-algebras, and as morphisms their $^{*}$-homomorphisms, i.e.\ the complex-algebra homomorphisms preserving the 
 involution and $1$. The  \emph{Gelfand duality theorem} asserts that the category $\Cst$ is dually equivalent to the category  of compact Hausdorff spaces and   continuous maps. We indicate how Gelfand duality fits in the framework of affine adjunctions developed above. The first important fact is that we can work at the level of the algebraic adjunction. For this, recall that $x\in A$ is \emph{self-adjoint} if it is fixed by $*$, i.e.\ if $x^*=x$. Further,  recall that self-adjoint elements carry a partial order which may be defined in several equivalent ways; see e.g.\ \cite[Section 8.3]{conway}. For our purposes here it suffices to define  a self-adjoint element $x\in A$ to be \emph{non-negative}, written $x\geq 0$, if there exists a self-adjoint $y\in A$ such that $x=y^2$. There is a functor  $U\colon \Cst\to\Set$  that takes a $C^*$-algebra $A$ to the collection of its non-negative self-adjoint elements whose norm does not exceed unity:
\[
U(A):=\{x\in A\mid x^*=x, 0\leq x, \|x\|\leq 1 \}.
\] 
In particular, writing $\Cx$ for the $C^*$-algebra on the complex numbers, $U(\Cx)=[0,1]$, the real unit interval. 
It is elementary that the restriction of a $^{*}$-homomorphism $A\to B$ to $U(A)$ induces a function $U(A)\to U(B)$,  so that $U$ is indeed a functor.

The following theorem puts together a number of known results; details and relevant references can be found in \cite{Marra-Reggio}.
\begin{theorem}\label{thm:negrepontis}The  category $\Cst$ is a variety with respect to the functor $U\colon \Cst\to\Set$.   This variety is not finitary, but can be presented using operations of countable arity. The left adjoint $F$ to the functor $U$  sends a set $S$  to the $C^*$-algebra of all complex valued, continuous functions on the compact Hausdorff space $[0,1]^{S}$.\end{theorem}
We set $\Va:=\Cst$, and $A:=\Cx$. Corollary \ref{cor:algadj} yields a dual adjunction 
 between $\Cst$ and the category of subsets of $[0,1]^{\mu}$ ---with $\mu$ ranging over all cardinals--- and definable maps.   
The characterisation of the fixed points of the adjunction on the algebraic side is now  similar to the one in Stone duality; the only specific piece of information we need on the class of $C^{*}$-algebras is the following.
\begin{lemma}\label{l:c*semisimple}
\mbox{}
\begin{enumerate}
\item\label{l:c*semisimple-item2}   The only simple $C^*$-algebra is $\Cx$.
\item\label{l:c*semisimple-item1} The variety $\Va^{*}$  semisimple.
\end{enumerate}
\end{lemma}
\begin{proof}
The first item amounts to the standard fact that a quotient of a $C^*$-algebra modulo an ideal $I$ is isomorphic to $\Cx$ if, and only if, $I$ is maximal. The second item amounts to the equally well-known fact that each $C^*$-algebra has enough maximal ideals to separate elements, and thus is a subdirect product of copies of $\Cx$.
\end{proof}

\begin{corollary}
Every  $C^{*}$-algebra is fixed by the composition $\C\circ \V$.
\end{corollary}
\begin{proof}
 Let  $F(\mu)/\theta$ be any presented algebra in $\Cst$. By Lemma \ref{l:c*semisimple} (item \ref{l:c*semisimple-item1}), $F(\mu)/\theta$ is semisimple, so by definition there is a subdirect embedding of $F(\mu)/\theta$ into a product of simple algebras. By Lemma \ref{l:c*semisimple} (item \ref{l:c*semisimple-item2}), each simple algebras in the decomposition is   isomorphic to $\Cx$.  So, by Lemma \ref{l:SGK}, each simple algebra is isomorphic to $F(\mu)/\CC(\{a\})$ for suitable $a$'s.   Since the decomposition is subdirect $\theta\subseteq \CC(\{a\})$, so by the Galois connection (\ref{eq:galois}) we have  $a\in \VV(\theta)$.  Thus we have a subdirect embedding $F(\mu)/\theta$ into $\prod_{a\in\VV(\{a\})}F(\mu)/\CC(\{a\})$, and this embedding is verified to coincide with the Birkhoff transform by the uniqueness property of the latter.  Thus, Theorem \ref{thm:algnull} applies to yield that  the congruence $\theta$ is fixed by $\CC\circ \VV$, from which the thesis follows at once.\end{proof}
We thus have:
\begin{corollary}
The category $\Cst$ is dually equivalent to the category of subsets of $[0,1]^{\mu}$ that are fixed by $\VV\circ\CC$, as $\mu$ ranges over all cardinals, and definable maps between them.
\end{corollary}
Passage from this affine adjunction to  classical Gelfand duality is conceptually analogous to what we have seen for Boolean algebras, but calls for deeper work. The characterisation of the fixed points on the spatial  side requires first to show that the operator $\VC$ is topological, then that the Zariski topology on $U(\Cx)=[0,1]$ is Hausdorff, and finally that each definable function $[0,1]^\mu\to [0,1]$ is continuous with respect to the product topology on $[0,1]^\mu$ and the Zariski topology on $[0,1]$.  Lemma \ref{l:co-null} then guarantees that the subsets of $[0,1]^\mu$ that are fixed by the adjunction are precisely the closed sets. Turning to morphisms, the definable maps must be proved to be precisely the continuous maps between such closed sets. We note that this fact essentially amounts to the Stone-Weierstrass theorem. Finally, the classical version of Gelfand duality is obtained through an intrinsic  characterisation of the closed  sets of $[0,1]^\mu$.  This is  the well-known theorem \cite[Lemma 4.5, p.\ 116]{kelley1955general}: A topological space is compact and Hausdorff if, and only if, it is homeomorphic to a closed subset of the Tychonoff cube $[0,1]^{\mu}$, for some $\mu$.
\subsection*{Acknowledgements.}
\noindent The first author thankfully acknowledges partial support from a Research Fellowship at Jesus College, Cambridge, a CARMIN Fellowship at IH\'ES-IHP, and a Marie Curie INdAM-COFUND-2012 Fellowship.
The second author gratefully acknowledges partial support by the Italian FIRB ``Futuro in Ricerca'' grant RBFR10DGUA.
The second and third  authors further  acknowledge  partial support from the Italian National Research Project (PRIN2010--11) entitled \emph{Metodi logici per il trattamento dell'informazione}.  Parts of this article where written while the second and third  authors were kindly hosted by the CONICET in Argentina within the European FP7-IRSES project \emph{MaToMUVI} (GA-2009-247584).  The third author gratefully acknowledges  partial support by the Marie Curie Intra-European Fellowship for the project ``ADAMS" (PIEF-GA-2011-299071). Finally, the second author wishes to express his gratitude to Dirk Hofmann for a useful conversation on the subject of this paper, and  for his most valuable bibliographic and mathematical suggestions related to the general theory of  concrete dual adjunctions.

\end{document}